\documentclass{amsart}
\usepackage{amscd}
\usepackage{amsfonts}
\usepackage{amsmath}
\usepackage{amssymb}
\usepackage{latexsym}
\usepackage{amsthm,color}
\usepackage[all]{xy}
\usepackage{epsfig}
\usepackage{graphicx}
\usepackage{color}
\newtheorem{proposition}{Proposition}
\newtheorem{lemma}{Lemma}[section]
\newtheorem{sublemma}{Sublemma}[section]
\newtheorem{theorem}{Theorem}

\newtheorem{fact}{Fact}

\theoremstyle{definition}
\newtheorem{definition}{Definition}
\theoremstyle{definition}
\newtheorem{example}{Example}
\theoremstyle{definition}
\newtheorem{remark}{Remark}
\newtheorem{note}{Note}

\newtheorem{problem}{Problem}

\newcommand{\R}{\mathbb{R}}

\newcommand{\mb}{\mathbb}
\DeclareMathOperator{\Reg}{Reg}

\usepackage{todonotes}

\newcommand{\ch}[1]{{\color{black} #1}}

\graphicspath{{figures/}}

\begin{document}
\title[
Recovery problem of parametrizations from Legendre data 
]
{Recovery problem of parametrizations from Legendre data  
}
\author[C.~Mu\~{n}oz-Cabello]{C.~Mu\~{n}oz-Cabello}
\address{C.~Mu\~{n}oz-Cabello \\ 
Departament de Matem\`{a}tiques, Universitat de Val\`{e}ncia, Campus de
Burjassot, 46100 Burjassot, Spain}
\email{Christian.Munoz@uv.es}
\author[T.~Nishimura]{T.~Nishimura
}
\address{T.~Nishimura \\
Research Institute of Environment and Information Sciences,
Yokohama National University,
Yokohama 240-8501, Japan}
\email{nishimura-takashi-yx@ynu.ac.jp}
\author[R.~Oset Sinha]{R.~Oset Sinha}
\address{R.~Oset Sinha \\
Departament de Matem\`{a}tiques, Universitat de Val\`{e}ncia, Campus de
Burjassot, 46100 Burjassot, Spain}
\email{Raul.Oset@uv.es}
\begin{abstract}
The problem of recovery 
of parametrizations from Legendre data is a very important inverse problem.    
In this paper,  
we provide a systematic and widely-applicable method to 
recover parametrizations 
$\ch{f}: {\color{black}U_n}\to \mathbb{R}^{n+1}$ from Legendre data 
where $U_n$ is an open subset of $\mathbb R^n$.  
Namely, for a dense subset of {\color{black}the space of real-analytic
parametrizations 
from $U_n$ into $\mathbb{R}^{n+1}$}, 
 we show how to recover 
the parametrization from the Gauss mapping and the height function.     
{\color{black} 
Moreover, in order to assist readers to apply results of this paper, 
many concrete examples are given.  
} 
\end{abstract}
\subjclass{57R45, 58C25} 
\keywords{Recovery problem, Envelope, Gauss mapping, Height function, 
Frontal,
Regular frontal, {\color{black}P}seudo regular frontal.}


\date{}

\maketitle
\section{Introduction}\label{section1}
Throughout this paper, unless otherwise stated, 
$n$ is a positive integer 
and $U_n$ is an open subset 
of $\mathbb{R}^n$.   Moreover, all functions and 
mappings are {\color{black}real-analytic}. 

Recovery problems are important inverse problems since the  
recovery of unprocessed raw data from processed data often gives a 
 breakthrough in Science and Engineering.     
\par 
The Legendre transform is one example of such a recovery.   
A brief explanation of the Legendre transform is as follows.   
For instance, let $f: \mathbb{R}\to \mathbb{R}$ be the function defined by 
$f(x)=x^2$ and consider $f$ to be the unprocessed data.     
Draw the graph of $f$ in the $(X, Y)$-plane.    
The tangent line to the graph of $f$ at $x=x_0$ is defined by 
\begin{eqnarray*}
Y & = & \frac{df}{dx}\left(x_0\right)\left(X-x_0\right)+x^2_0 \\ 
{ } & = & 2x_0X-x^2_0.   
\end{eqnarray*}
Set $p{\color{black}=\frac{df}{dx}(x_0)}=2x_0$ and 
$q= {\color{black}-\left(-x_0^2\right)=x^2_0}$.    
Then, we can obtain the new function 
\[
q(p)= {\color{black}\left(\frac{p}{2}\right)^2}.   
\]
Namely, the $Y$-intercept ${\color{black}-q}$ 
of the tangent line to the graph of $f$ 
is a function of the slope $p$ of the tangent line to the graph of $f$.   
The function $q(p)$ is a processed data of $f(x)$ and the process 
$\mathcal{L}(f{\color{black}(x)})= q{\color{black}(p)}$ 
is called the \textit{Legendre transform}.    
The recovery problem in this case is to obtain the original raw data $f(x)$ 
from the processed data $q(p)$. It is well-known that the same process gives 
the perfect solution of this recovery problem, that is, 
$\mathcal{L}(q{\color{black}(p)})=f{\color{black}(x)}$.     
Nowadays, Legendre transform has been widely extended 
and has important applications in Mathmatical Physics, 
Thermodynamics, Convex analysis etc.      
For more details on the Legendre Transform, see for instance 
{\color{black}\cite{arnoldmechanics, hormander, rockafellar}}.    
\par 
On the other hand,  consider for example the function 
$g: \mathbb{R}\to \mathbb{R}$ defined by $g(x)=x^3$.    
Draw the graph of $g$ in the $(X, Y)$-plane.    
The tangent line to the graph of $g$ at $x=x_0$ is defined by 
\begin{eqnarray*}
Y & = & \frac{dg}{dx}\left(x_0\right)\left(X-x_0\right)+x^3_0 \\ 
{ } & = & 3x_0^2X-2x^3_0.   
\end{eqnarray*}      
If we set $p=\frac{dg}{dx}\left(x_0\right)=3x_0^2$, the $Y$-inter\ch{c}ept 
${\color{black}-q}=-2x_0^3$ is not a function of $p$.    
Thus, it is impossible to apply the Legendre transform to obtain the original 
unprocessed data $g(x)=x^3$ from the processed data 
$\left\{(p(x), q(x))\right\}_{x\in \mathbb{R}}=
\left\{\left(3x^2, {\color{black}2x^3}\right)\right\}_{x\in \mathbb{R}}$.     
{\color{black}Nevertheless,} very recently, in \cite{nishimura, nishimura2}, by 
replacing the processed data 
$\left\{(p(x), q(x))\right\}_{x\in \mathbb{R}}=
\left\{\left(3x^2, {\color{black}2x^3}\right)\right\}_{x\in \mathbb{R}}$ 
with the processed data in the sense of Legendre (called \textit{Legendre data}) 
$\left\{\left(\nu(x), a(x)\right)\right\}_{x\in \mathbb{R}}$, 
the second author provides a new method to solve the recovery problem.     
A brief explanation of the new method 
is as follows.   
For the given $g(x)=x^3$, let  
$\nu: \mathbb{R}\to S^1$ be the mapping defined by 
$\nu(x)=\frac{1}{\sqrt{1+9x^4}}\left(-3x^2, 1\right)$.    
Then, $\nu\left(x_0\right)$ is a unit normal vector {\color{black}to} the 
tangent line to the graph of $g(x)=x^3$ at $x=x_0$ and the mapping 
$\nu: \mathbb{R}\to S^1$ is called the \textit{Gauss mapping}.  
Set $a(x)=\left(x, g(x)\right)\cdot \nu(x)=\frac{-2x^3}{\sqrt{1+9x^4}}$, where 
the dot in the center stands for the standard scalar product of 
$2$-dimensional vectors. Then,  $a(x_0)$ is the height of the tangent 
line to the graph of $g(x)=x^3$ at $x=x_0$ relative to the origin $(0,0)$ and 
the function $a: \mathbb{R}\to \mathbb{R}$ is called the 
\textit{height function}.      
Thus, the Legendre data in this case is 
$\left\{\frac{1}{\sqrt{1+9x^4}}\left(-3x^2, 1\right), \; 
\frac{-2x^3}{\sqrt{1+9x^4}}\right\}_{x\in \mathbb{R}}$.    
Next, set $\nu(x)=\left(\cos\theta(x), \sin\theta(x)\right)$.    
It has been shown in \cite{nishimura, nishimura2} that 
$\frac{da}{dx}(x)$ is divided by $\frac{d\theta}{dx}(x)$ and 
by setting $b(x)=\frac{\frac{da}{dx}(x)}{\frac{d\theta}{dx}(x)}$, 
the original unprocessed data $g(x)$ can be actually recovered as follows.   
\[
{\color{black} 
\left(x, g(x)\right)
}
=a(x)\left(\cos\theta(x), \sin\theta(x)\right) + 
b(x)\left(-\sin\theta(x), \cos\theta(x)\right).  
\]  
In the case that $g(x)=x^3$, by elementary calculation, it follows that
$b(x)=\frac{-x-3x^5}{\sqrt{1+9x^4}}$. Thus, as desired, 
we certainly have 
\begin{eqnarray*}
{ } & { } & a(x)\left(\cos\theta(x), \sin\theta(x)\right) + 
b(x)\left(-\sin\theta(x), \cos\theta(x)\right) \\ 
{ } & = & 
\frac{-2x^3}{\sqrt{1+9x^4}}\left(\frac{-3x^2}{\sqrt{1+9x^4}}, 
\frac{1}{\sqrt{1+9x^4}}\right)
+
\frac{-x-3x^5}{\sqrt{1+9x^4}}\left(\frac{-1}{\sqrt{1+9x^4}}, 
\frac{-3x^2}{\sqrt{1+9x^4}}\right) \\ 
{ } & = & 
\left(\frac{6x^5}{{1+9x^4}}, 
\frac{-2x^3}{{1+9x^4}}\right)
+
\left(\frac{x+3x^5}{{1+9x^4}}, 
\frac{3x^3+9x^7}{{1+9x^4}}\right) \\ 
{ } & = & 
\left(\frac{x+9x^5}{1+9x^4}, \frac{x^3+9x^7}{1+9x^4}\right) \\ 
{ } & = & 
\left(x, x^3\right)=
{\color{black} 
\left(x, g(x)\right)
}.      
\end{eqnarray*}
The new method developed in \cite{nishimura, nishimura2} is based on the study 
of anti-orthotomics of frontals 
(\cite{janeczkonishimura}, see also \cite{alamocriado} in which 
the usefulness of anti-orthotomics for the recovery problem 
of wave fronts has been {\color{black}already} clarified).      
Moreover, the new method 
may be regarded as a natural generalization of the outstanding 
Cahn-Hoffman vector formula (see \cite{hoffmancahn}.  
{\color{black}see also 
\cite{hedgehog} in which hedgehogs were defined as natural geometric 
objects whose parametrizations can be de\ch{sc}ribed by using the Cahn-Hoffman 
vector formula}). 
\par 
\smallskip 
In this paper, we {\color{black}sublimate the new method 
in \cite{nishimura, nishimura2} to solve the recovery problem \ch{even} 
in higher dimensional cases.   \ch{As} 
it turns out, \ch{our method solves} 
the recovery problem for 
any {\color{black}real-analytic} 
parametrisation $f\colon U_n \to \mathbb{R}^{n+1}$ 
from Legendre data {\color{black}(Theorem \ref{theorem3})}.    
Thus, our sublimation 
is much \ch{stronger than expected}.}   
{\color{black}More precisely, we first define the notion of regular frontal 
and we show the recovery problem for any regular frontal can be 
directly solved by our sublimation {\color{black}(Theorem \ref{theorem1})},    
Secondly, we generalize the notion of regular frontal to} 
the notion of pseudo regular frontal, for which the 
{\color{black}sublimation} can be directly applied 
{\color{black}as well (the assertion (1) of Theorem \ref{theorem2}).    
Moreover, w}e show the space of pseudo regular frontals is dense 
in the space of all {\color{black}real-analytic} mappings 
{\color{black}(the assertion (2) of Theorem \ref{theorem2})}, 
and thus, for {\color{black}any} non-pseudo regular frontal, 
we can recover the parametrisation as the limit of a sequence of 
pseudo regular frontals {\color{black}(Theorem \ref{theorem3})}. 
\par 
\medskip 
The paper is organized as follows.   
In Section \ref{section2} the main results 
(Theorem \ref{theorem1}, Theorem \ref{theorem2}, Theorem \ref{theorem3}) 
shall be stated.   
Theorem \ref{theorem1}, Theorem \ref{theorem2} and Theorem \ref{theorem3}  
shall be proved in Section \ref{section3}, Section \ref{section4} 
and Section \ref{section5}, respectively.     
Finally, in Section \ref{examples}, many concrete examples with precise 
calculations shall be given.   

\section{Main results}\label{section2}
\begin{definition}\label{frontal}
{\rm
A {\color{black}real-analytic} mapping $f: U_n\to \mathbb{R}^{n+1}$ is called 
a \textit{frontal} if there exists a real-analytic mapping 
$\nu: U_n\to S^n$ such that 
$df_{\color{black}\bf x}({\color{black}\bf v}){\color{black}\cdot} 
\nu({\color{black}\bf x})=0$ for any 
${\color{black}\bf x}\in U_n$ and any ${\color{black}\bf v}\in 
T_{\color{black}\bf x}U_n$.    
Here, $S^n$ is the unit $n$-dimensional sphere 
{\color{black}in $\mathbb{R}^{n+1}$}.  
}
\end{definition}
\noindent 
For a frontal $f: U_n\to \mathbb{R}^{n+1}$, the mapping 
$\nu: U_n\to S^n$ satisfying Definition \ref{frontal} 
(resp., the function $a: U_n\to \mathbb{R}$ defined by  
$a({\color{black}\bf x})=f({\color{black}\bf x}) 
{\color{black}\cdot} \nu({\color{black}\bf x})$)  
is called the \textit{Gauss mapping} (resp., the \textit{height function}) 
of $f$.    
{\color{black} 
For readers who are not familiar with frontals, 
\cite{ishikawa1, ishikawa2} are recommended as excellent overviews 
with clear explanations.     
\par 
Given a frontal $f: U_n\to \mathbb{R}^{n+1}$, by 
the Gauss mapping $\nu: U_n\to S^n$ and the height function 
$a: U_n\to \mathbb{R}$, the hy\ch{p}erplane family 
$\mathcal{H}_{(\nu, a)}$ can be naturally defined as follows.   
\begin{eqnarray*} 
H_{(\nu({\bf x}), a({\bf x}))} & = & 
\left\{{\bf X}\in \mathbb{R}^{n+1}\, |\, {\bf X}\cdot 
\nu({\bf x})=a({\bf x})\right\} \\ 
\mathcal{H}_{(\nu, a)} & = & \left\{H_{(\nu({\bf x}), a({\bf x}))}\right\}_{{\bf x}\in U_n}.  
\end{eqnarray*}
Hence, the data set  
$\left\{\nu({\bf x})\in S^n, \; a({\bf x})\in \mathbb{R}\right\}_{{\bf x}\in U_n}$ 
is called the 
{\it Legendre data} of the frontal $f$.    
\begin{example}\label{example0}
Let $\nu: \mathbb{R}\to S^1$ 
be the constant mapping \ch{given by} 
$\nu({x})=(0,1)\in S^1$ and  $a: \mathbb{R}\to \mathbb{R}$ \ch{be the function} $a(x)=x$.    Then, it is clear that there \ch{is} 
no frontal $f$ such that $\{(\nu(x) , a(x))\}_{x\in \mathbb{R}}$ 
is the Legendre data of $f$.    
\end{example} 
\noindent 
By Example \ref{example0}, not all data sets   
$\left\{\nu({\bf x})\in S^n, \; a({\bf x})\in \mathbb{R}\right\}_{{\bf x}\in U_n}$ 
can become the Legendre data of a certain frontal $f$.     
\begin{definition}\label{definition1}
Let $\mathcal{H}_{\left(\nu, a\right)}$ 
be a hyperplane family.   
A real-analytic mapping ${f}: U_n\to \mathbb{R}^{n+1}$ is called 
an \textit{envelope created by} 
$\mathcal{H}_{\left(\nu, a\right)}$ 
if the following two conditions are satisfied.   
\begin{enumerate}
\item[(a)] $f({\bf x})\in 
H_{\left(\nu({\bf x}), a({\bf x})\right)}$ 
for any ${\bf x}\in U_n$.  
\item[(b)] $d{f}_{\bf x}({\bf v})\cdot \nu({\bf x})=0$ 
for any ${\bf x}\in U_n$ and 
any ${\bf v}\in T_{\bf x} U_n$.    
\end{enumerate}
\end{definition}
\noindent 
In other words, an {envelope created by} 
$\mathcal{H}_{\left(\nu, a\right)}$ 
is a mapping $f: U_n\to \mathbb{R}^{n+1}$ giving a solution  
of the following system of first order differential equations with one constraint 
condition, where $\left(U, \left(x_1, \ldots, x_n\right)\right)$ 
is an arbitrary coordinate neighborhood of $U_n$ such that 
${\bf x}\in U\subset U_n$.      
\[
\left\{
\begin{array}{ccc}
\frac{\partial f}{\partial x_1}({\bf x})\cdot 
\nu({\bf x}) & = & 0, \\ 
 \vdots & { } & { } \\
\frac{\partial f}{\partial x_n}({\bf x})
\cdot \nu({\bf x}) & = & 0, \\ 
f({\bf x})\cdot \nu({\bf x}) & = & a({\bf x}).     
\end{array}
\right.
\]  
\par 
\noindent 
For details on envelopes created by families of plane regular curves, 
refer to \ch{the} excellent book 
\cite{brucegiblin}.    
By definition, any frontal $f: U_n\to \mathbb{R}^{n+1}$ is an envelope created 
by a hyperplane family 
$\mathcal{H}_{(\nu, a)}$ where $\nu: U_n\to S^n$ 
(resp., $a: U_n\to \mathbb{R}$) is the Gauss mapping (resp., the height function) 
of $f$.     Conversely, again by definition, 
any envelope $f: U_n\to \mathbb{R}^{n+1}$ 
created by a hyperplane family 
$\mathcal{H}_{\left(\nu, a\right)}$ 
is a frontal with the Gauss mapping $\ch{\nu}: U_n\to S^n$ and the 
height function $a: U_n\to \mathbb{R}$.     
Thus, frontals and envelopes of hyperplane families are exactly  
the same notion.     
Nevertheless, we prefer to study  
the notion of envelopes created by hyperplane families more 
becuse it seems that envelopes created by hyperplne families are 
more applicable (for instance, see \cite{quantum_entanglement}).   
As in E\ch{xa}mple \ref{example0}, not all hyperplane families create 
envelopes.     Thus, it is important to solve the following \ch{item} (1) of 
Problem \ref{problem1}.    
In addition, since envelopes are real-analytic mappings by definition, 
especially for applications of envelopes created by hyperplane families, 
solving the following \ch{item} (2) of Problem \ref{problem1} is important as well.    
\begin{problem}\label{problem1}
Let $\nu: U_n\to S^n$ and $a: U_n\to \mathbb{R}$ 
be a real-analytic mapping and a real- analytic function respectively.   
\begin{enumerate}
\item[(1)] When and only when does the hyperplane family 
$\mathcal{H}_{(\nu, a)}$ create an envelope? \\ 
\item[(2)] Suppose that the hyperplane family $\mathcal{H}_{(\nu, a)}$ creates 
an envelope $f: U_n\to \mathbb{R}^{n+1}$.     
Then, describe $f$ in terms of $\nu, a$.    
\end{enumerate}
\end{problem}
Both problems in Problem \ref{problem1} \ch{have} been solved as follows.   
\begin{fact}[\cite{nishimura}]\label{fact0}
Let $\nu: U_n\to S^n$ and $a: U_n\to \mathbb{R}$ 
be a real-analytic mapping and a real- analytic function respectively. 
\begin{enumerate}
\item[(1)] The hyperplane family 
$\mathcal{H}_{\left(\nu, a\right)}$ 
creates an envelope if and only if 
there exists a $1$-form {$\Omega$} along 
${\nu}$ such that for any ${\bf x}_0\in U_n$ 
by using of a coordinate neighborhood 
$\left(U, \left(x_1, \ldots, x_n\right)\right)$ of $U_n$ at 
${\bf x}_0$ and 
a \textit{normal} coordinate neighborhood 
$\left(V, \left(\Theta_1, \ldots, 
\Theta_n\right)\right)$ 
of  $S^n$ at  
${\nu}\left({\bf x}_0\right)$, 
the $1$-form germ 
$d{\gamma}$ at ${\bf x}_0$ 
is expressed 
as follows.   
\[
d{a} = 
\sum_{i=1}^n \left(\omega({\bf x})
\left({\Pi}_{\left({\nu}({\bf x}), {\nu}({\bf x}_0)\right)} 
\left(\frac{\partial}{\partial \Theta_i}\right) \right) \right)
d \left(\Theta_i\circ{\nu}\right), 
\]
where {
a normal coordinate neighborhood   
$\left(V, \left(\Theta_1, \ldots, 
\Theta_n\right)\right)$ is a 
nothing but a local inverse mappping of 
the exponential mapping at 
${\nu}\left({\bf x}_0\right)$ and   
${\color{black}\Pi}_{\left({\nu}({\bf x}), {\nu}({\bf x}_0)\right)} : 
T_{{\nu}\left({\bf x}_0\right)}S^n\to T_ {{\nu}\left({\bf x}\right)}S^n$ 
is the Levi-Civita translation.      
{Notice that since the unit sphere $S^n$ with metric inherited from 
$\mathbb{R}^{n+1}$, the Levi-Civita translation 
$\Pi_{\left({\nu}({\bf x}), {\nu}({\bf x}_0)\right)}$ 
is merely the restriction of the rotation 
$R: \mathbb{R}^{n+1}\to \mathbb{R}^{n+1}$ 
satisfying $R({\nu}({\bf x}_0))={\nu}({\bf x})$ 
to the tangent space $T_{{\nu}\left({\bf x}_0\right)}S^n$.    
In particular, in the case $n=1$, a normal coordinate $\Theta$ at 
${\nu}\left(x\right)$ is just the \textit{radian} 
(or, its negative) between two unit vectors ${\nu}\left({\bf x}_0\right)$ and 
${\nu}\left({\bf x}\right)$ and the Levi-Civita translation 
$\Pi_{\left({\nu}({\bf x}), \widetilde{\nu}({\bf x}_0)\right)}$ is nothing but  
the restriction of the plane rotation through $\Theta$ to the tangent space 
$T_{\widetilde{\nu}\left({\bf x}_0\right)}S^1$}.     \\ 
\item[(2)] Suppose that the hyperplane family $\mathcal{H}_{(\nu, a)}$ creates 
an envelope $f: U_n\to \mathbb{R}^{n+1}$.     
Then, for 
any ${\bf x}\in U_n$, 
under the canonical identifications 
$ 
T^*_{{\nu}({\bf x})}S^n\cong T_{{\nu}({\bf x})}S^n 
\subset T_{{\nu}({\bf x})}\mathbb{R}^{n+1} \cong \mathbb{R}^{n+1}$, 
the $(n+1)$-dimensional vector ${f}({\bf x})$ is represented as follows.   
\[
{f}({\bf x})={\omega}({\bf x})+{a}({\bf x}){\nu}({\bf x}),   
\]  
where the $(n+1)$-dimensional vector ${\omega}({\bf x})$ is identified 
with the corresponding $n$-dimensional 
cotangent vector ${\omega}({\bf x})$ under these identifications.   
}   
\end{enumerate}
\end{fact}
\noindent 
For readers who are not familiar with normal coordinate neighborhoods,  
exponential mappings and Levi-Civita translations, 
\cite{kobayashinomizu} is recommended 
as one of excellent references.   
\par 
The recovery problem from Legendre data which we \ch{deal with} in this paper  
arises as a hard problem when we try to \ch{do concrete computations} by using 
the assertion (2) of Fact \ref{fact0} in the case 
$n\ge 2$.    
In the case $n=1$, there \ch{is} no problem on \ch{doing} concrete computation\ch{s} as in 
Section \ref{section1}.     
On the other hand, in the case $n\ge 2$, even by using 
symbolic algebras, it seems that local concrete computations on the one form 
$\omega({\bf x})$ and the Levi-Civita translation 
${\color{black}\Pi}_{\left({\nu}({\bf x}), {\nu}({\bf x}_0)\right)} : 
T_{{\nu}\left({\bf x}_0\right)}S^n\to T_ {{\nu}\left({\bf x}\right)}S^n$ 
are almost impossible.    
Thus, for the recovery of the concrete form of \ch{the} parametrization $f$ 
from \ch{the} Legendre data $\left\{\nu({\bf x}), a({\bf x})\right\}_{{\bf x}\in U_n}$ 
in the case $n\ge 2$, we had to exploit a new technique 
instead of assertion (2) of Fact \ref{fact0}.    
}   
\par 
\smallskip 
Notice that any non-singular mapping 
$f: U_n\to \mathbb{R}^{n+1}$ is 
a frontal.     
It is T.~Banchoff, T.~Gaffney and C.~McCrory who started to study 
the Gauss mappings $\nu: U_2\to S^2$ (see \cite{cusps}).    
Their study was concentrated on cusps of Gauss  mappings 
for non-singular surfaces 
{\color{black}
$f: U_2\to \mathbb{R}^3$
}.   
On the other hand,  as in the following Example \ref{example1}, 
there are many singular frontals.      
Notice that the Gauss mappings for singular frontals 
$f: U_n\to \mathbb{R}^{n+1}$ have not been studied in depth {\color{black}so far}.     
One of the purposes of this paper is to 
provide an essential contribution on studying Gauss mappings 
for singular frontals.       
\begin{definition}\label{regular frontal}
A frontal $f: U_n\to \mathbb{R}^{n+1}$ is called 
a \textit{regular frontal} if for
the Gauss mapping $\nu:U_n\to S^n$ of $f$, the set 
\[ 
Reg(\nu)=\left\{{\color{black}\bf x}=\left(x_1, \ldots, x_n\right)\in U_n \; |\; 
{\color{black}\bf x} \mbox{ is a regular point of }\nu\right\}
\] 
is 
dense in $U_n$.   
\end{definition} 
 
Notice that $Reg(\nu)$ is always open since 
$Reg(\nu)$ may be characterized as the inverse image of an open subset 
$\mathbb{R}-\{0\}$ (of $\mathbb{R}$) by the Jacobian determinant of $\nu$ which 
is a continuous function $: U_n\to \mathbb{R}$.   

\begin{definition}\label{pseudo regular frontal}
{\rm 
A {\color{black}real-analytic} mapping $f: U_n\to \mathbb{R}^{n+1}$ is called 
a \textit{pseudo regular frontal} if the following (1), (2) are satisfied.   
\begin{enumerate}
\item[(1)] The set 
\[
Reg(f)=\left\{{\color{black}\bf x}=\left(x_1, \ldots, x_n\right)\in U_n\; |\; 
{\color{black}\bf x} \mbox{ is a regular point of }f\right\}
\]
is 
dense in $U_n$.   
\item[(2)] For {\color{black}a} Gauss mapping $\nu: Reg(f)\to S^n$ 
of {\color{black}$f|_{Reg(f)}$}, the set 
\[ 
Reg(\nu)=\left\{{\color{black}\bf x}=\left(x_1, \ldots, x_n\right)\in Reg(f) \; |\; 
{\color{black}\bf x} \mbox{ is a regular point of }\nu\right\}
\] 
is 
dense in $Reg(f)$.   
\end{enumerate}
}
\end{definition} 

Notice that similarly as $Reg(\nu)$, the set $Reg(f)$ is always open.   
\begin{note}
\begin{enumerate}
\item[(1)] Any regular frontal is a frontal, but its converse is not true in general 
(see (2), (4) of Example \ref{example1} below).   \\ 
\item[(2)] Any regular frontal is a pseudo regular frontal, 
but its converse is not true in general (see (5) of Example \ref{example1} below).  
\end{enumerate}
\end{note}

\begin{example}\label{example1}  
For detailed analysis of the following examples, see Section \ref{examples}.
\begin{enumerate}
\item[(1)] The mapping $f: U_1\to \mathbb{R}^2$ defined by 
\[
f(x)=\left(x^2, x^3\right)
\]
is called the \textit{standard normal form of the cusp}, and it is an example of frontal and of regular frontal as well. 
 \item[(2)] The mapping $f: U_2\to \mathbb{R}^3$ defined by 
\[
f(x,y)=\left(x, y^2, y^3\right)
\]
is called the \textit{standard normal form of the cuspidal edge}, 
and it is an example of frontal but {not a regular frontal}. Any $f$ equivalent under diffeomorphisms in source and target to $f$ is a called a \textit{cuspidal edge}. The mapping $f: U_2\to \mathbb{R}^3$ defined by 
\[
f(x,y)=\left(x, y^2+x^2, y^3\right)
\]    
is a cuspidal edge which is a frontal and a regular frontal as well.
 \item[(3)] The mapping $f: U_2\to \mathbb{R}^3$ defined by 
\[
f(x, y)=\left(x, y^2, xy^3\right)
\]
is called the \textit{standard normal form of the cuspidal crosscap}, 
and it is an example of 
frontal and {of regular frontal} as well.     
\item[(4)] The mapping $f: U_2\to \mathbb{R}^3$ defined by 
\[
f(x, y)=\left(x, 4y^3+2xy, 3y^4+xy^2\right)
\]
is called the \textit{standard normal form of the swallowtail}, 
and it is an example of 
frontal but {not a regular frontal}.     
\item[(5)] The mapping $f: \mathbb{R}^2\to \mathbb{R}^3$ defined by 
\[
f(x, y)=\left(x, y^2, xy\right)
\]
is called the \textit{standard normal form of the crosscap}, 
and it is not a frontal.    Set $U_2=\mathbb{R}^2-\{(0,0)\}$.    
Then, $f|_{{U_2}}$ is \text{\color{black}non-singular}.    
{It is easily seen that $f|_{{U_2}}$ is a regular  
frontal} as well.    
{Thus, $f$ is a pseudo regular frontal.}     
\end{enumerate}
The image of these mappings can be found \ch{in} Figure \ref{fig ex1}.
\end{example} 

\begin{figure}[ht]
	\includegraphics[width=.3\textwidth]{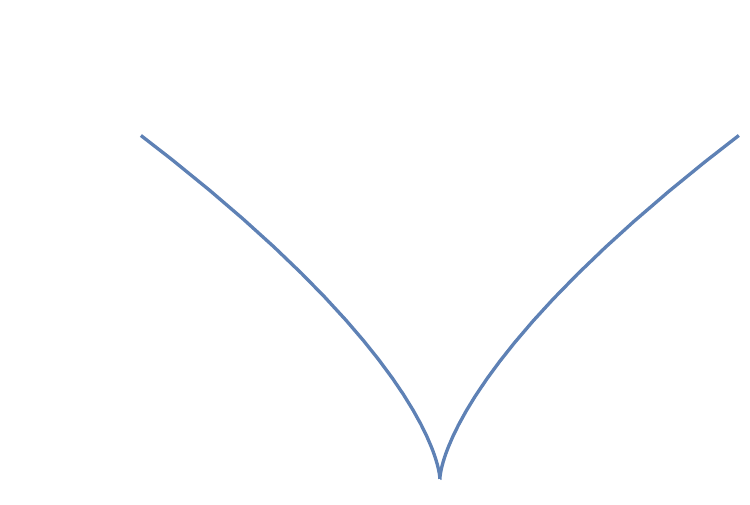}
		\hspace{1.5em}
	\includegraphics[width=.3\textwidth]{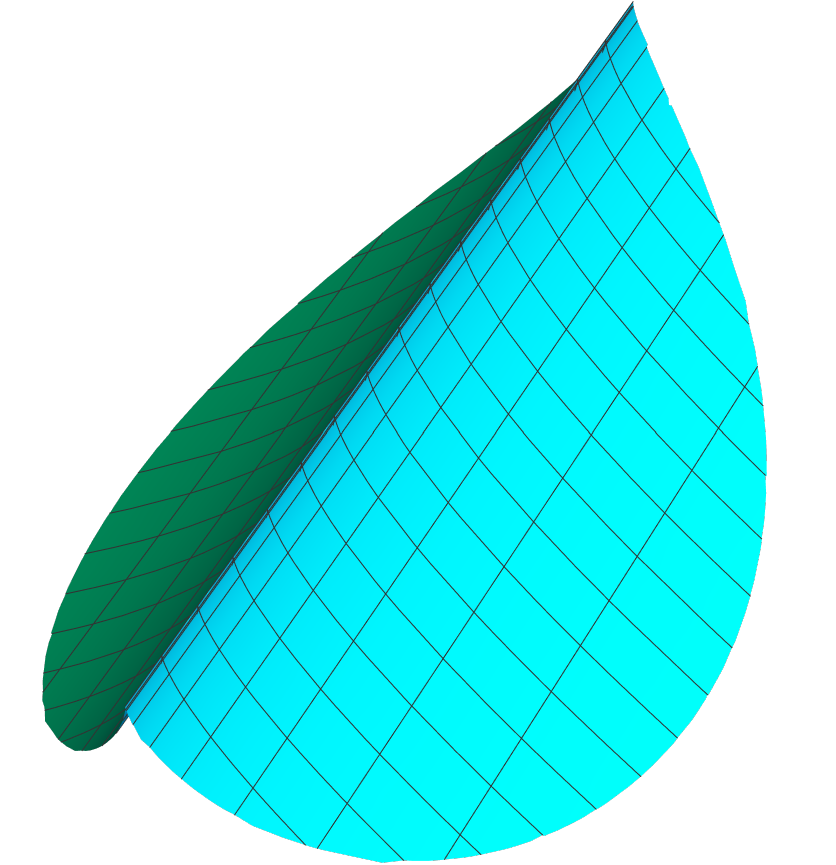}
		\hspace{1.5em}
	\includegraphics[width=.3\textwidth]{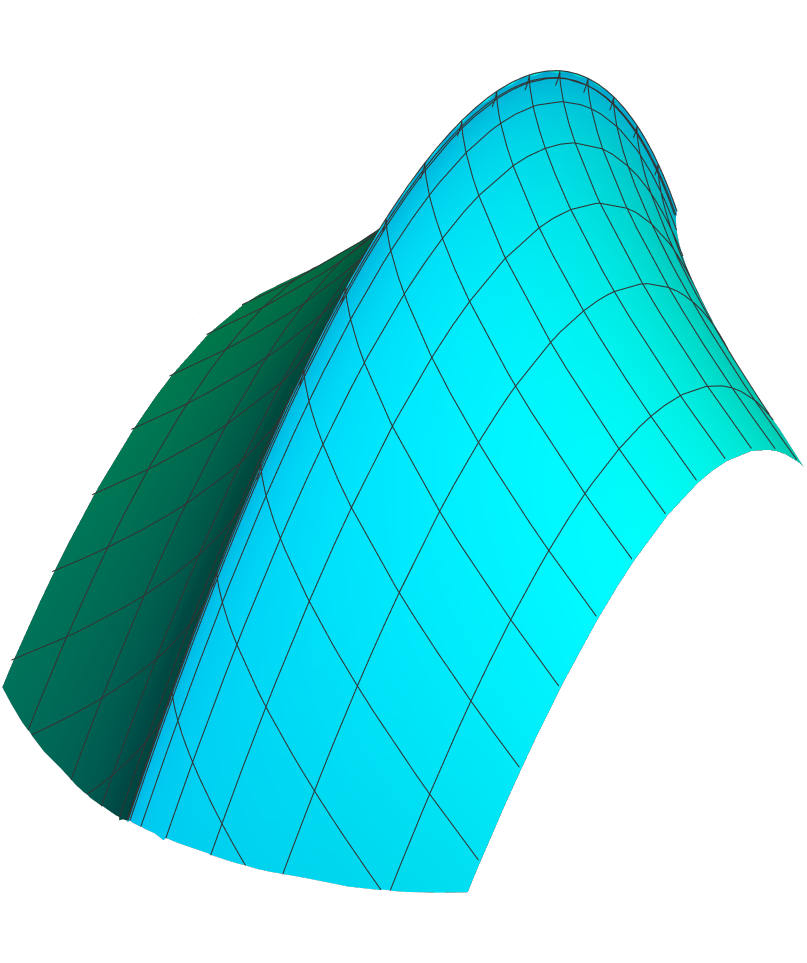}\\
	\includegraphics[width=.3\textwidth]{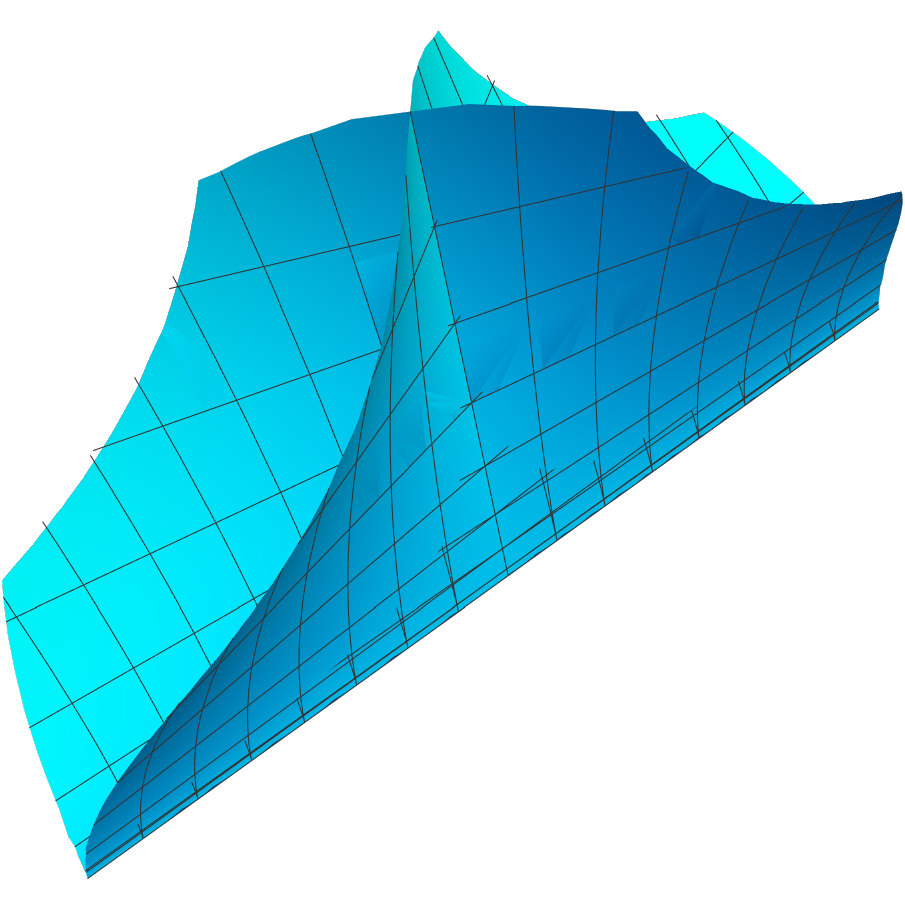}
		\hspace{.5em}
	\includegraphics[width=.3\textwidth]{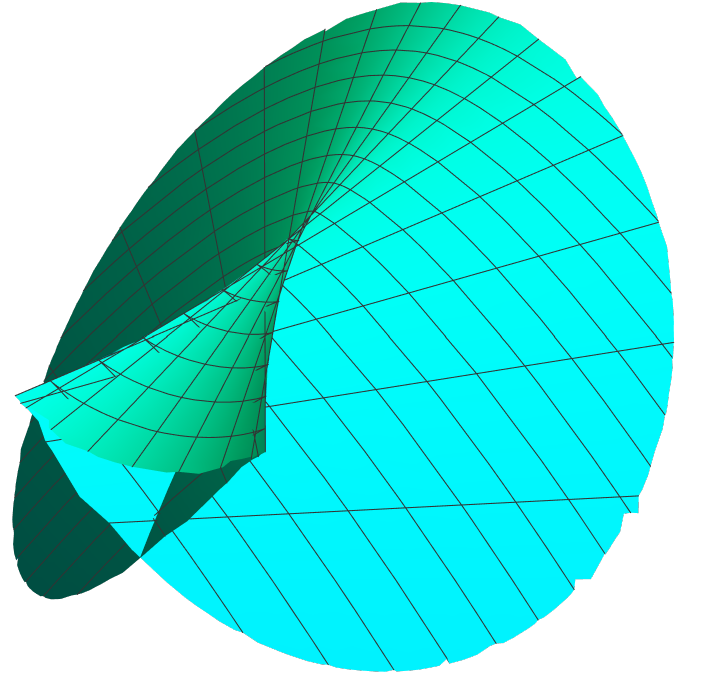}
		\hspace{1.5em}
	\includegraphics[width=.3\textwidth]{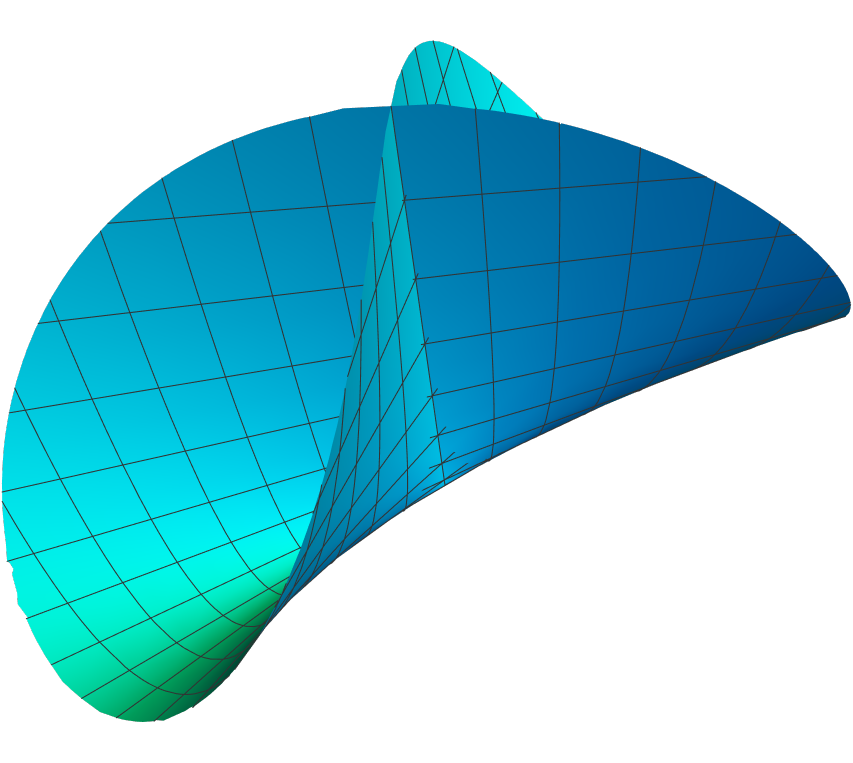}
	\caption{Images of the mappings from Example \ref{example1}, in order of appearance. \label{fig ex1}}
\end{figure}

\begin{definition}\label{Legendre data}
{\rm 
Given a {\color{black}pseudo regular} frontal $f: U_n\to \mathbb{R}^{n+1}$, 
the set 
$
\left\{
\nu({\color{black}\bf x})
\; 
a({\color{black}\bf x})
\right\}_{{\color{black}\bf x}\in Reg({\color{black}\nu})}
$ 
is called the \textit{Legendre data} of $f$,\; 
where $\nu: Reg(f)\to S^n$ {\color{black}is a Gauss mapping of $f|_{Reg(f)}$,  
$a: Reg(f)\to \mathbb{R}$ is the height function of $f|_{Reg(f)}$ and 
$Reg(\nu)$ is the set of regular points of $\nu$
}.       
}
\end{definition} 
\begin{theorem}\label{theorem1}
Let $f: U_n\to \mathbb{R}^{n+1}$ be a {regular} frontal.   
Then, $f$ is recovered from its Legendre data 
$\left\{\nu({\color{black}\bf x}), 
a({\color{black}\bf x})\right\}_{{\color{black}\bf x}\in Reg({\color{black}\nu})}$.   
{\color{black} 
Namely, for any ${\color{black}\bf x}\in U_n$, 
$f({\color{black}\bf x})$ can be concretely described in terms of 
$\nu\left({\color{black}\bf x}_i\right), 
a\left({\color{black}\bf x}_i\right)$ and $\lim_{i\to \infty}$ where 
$\left\{{\color{black}\bf x}_i\ldots\right\}_{i=1, 2, \ldots}$ 
is any sequence of regular points of $\nu: U_n\to S^n$ 
such that $\lim_{i\to \infty}{\color{black}\bf x}_i={\color{black}\bf x}$.     
}
\end{theorem}
Let {\color{black}$A(U_n, \mathbb{R}^{n+1})$} be the topological space 
of {\color{black}real-analytic} 
mappings $f: U_n\to \mathbb{R}^{n+1}$ endowed 
with the
Whitney $C^\infty$ topology.   
Let ${PRF}\left(U_n, \mathbb{R}^{n+1}\right)$ be the topological 
subspace of {\color{black}$A(U_n, \mathbb{R}^{n+1})$}, 
consisting of {pseudo regular} frontals.    
\begin{theorem}\label{theorem2}
{\color{black} 
As for the subspace of pseudo regular frontals 
${PRF}\left(U_n, \mathbb{R}^{n+1}\right)$, the following two hold.   
\begin{enumerate}
\item[(1)] Any $f\in {PRF}\left(U_n, \mathbb{R}^{n+1}\right)$ 
can be recovered from the Legendre data 
$\left\{\nu({\color{black}\bf x}),\; 
a({\color{black}\bf x}) \right\}_{{\color{black}\bf x}\in Reg(\nu)}$.  
{\color{black} 
Namely, for any ${\color{black}\bf x}\in U_n$, 
$f({\color{black}\bf x})$ can be concretely described in terms of 
$\nu\left({\color{black}\bf x}_i\right), 
a\left({\color{black}\bf x}_i\right)$ and $\lim_{i\to \infty}$ where 
$\left\{{\color{black}\bf x}_i\ldots\right\}_{i=1, 2, \ldots}$ 
is any sequence of regular points of $\nu: Reg(f)\to S^n$ 
such that $\lim_{i\to \infty}{\color{black}\bf x}_i={\color{black}\bf x}$.     
}
\item[(2)] The subspace ${PRF}\left(U_n, \mathbb{R}^{n+1}\right)$ is dense 
in $A(U_n, \mathbb{R}^{n+1})$.   
\end{enumerate}
}
\end{theorem}
\begin{theorem}\label{theorem3}
For any {\color{black}real-analytic} mapping 
$f: U_n\to \mathbb{R}^{n+1}$, let 
$\left\{f_1, f_2, \ldots\right\}$ be a sequence of {pseudo regular} frontals 
such that $\lim_{i\to \infty}{f_i}=f$.    Then, $f$ can be recovered from the 
sequence of Legendre data 
$ 
\left\{
\left\{
\nu_i({\color{black}\bf x}), a_i({\color{black}\bf x})
\right\}_{{\color{black}\bf x}\in Reg({\color{black}\nu}_i)}
\right\}_{i=1, 2, \ldots}$ of $f_i$.       
{\color{black} 
Namely, for any ${\color{black}\bf x}\in U_n$, 
$f({\color{black}\bf x})$ can be concretely described in terms of 
$\nu_i\left({\color{black}\bf x}_j\right), 
a_i\left({\color{black}\bf x}_j\right)$, $\lim_{i\to \infty}$ and 
$\lim_{j\to \infty}$ where 
$\left\{{\color{black}\bf x}_j\ldots\right\}_{j=1, 2, \ldots}$ 
is any sequence of regular points of $\nu_i: Reg(f_i)\to S^n$ 
such that $\lim_{j\to \infty}{\color{black}\bf x}_j={\color{black}\bf x}$.     
}
\end{theorem}
\section{Proof of Theorem 1}\label{section3} 

Let $\nu^n\colon W=\ch{\underbrace{]-\pi,\pi[\times \cdots \times ]-\pi,\pi[}_{n\text{-tuples}}}\to S^n$ be the mapping defined by
	\[\nu^n({\color{black}\boldsymbol{\theta}})=\left(\prod_{i=1}^n\cos \theta_i,\sin \theta_n\prod_{i=1}^{n-1}\cos \theta_i,\sin \theta_{n-1} \prod^{n-2}_{i=1}\cos \theta_i,\dots,\sin \theta_2 \cos \theta_1,\sin \theta_1\right).\]
This mapping is a homeomorphism on its image, thus defining a parametrisation of the $n$-sphere.
If we identify $\mathbb{R}^{n+1}$ with $T_{\nu^n({\color{black}\boldsymbol{\theta}})}\mathbb{R}^{n+1}$ 
for some ${\color{black}\boldsymbol{\theta}} \in W$, 
it follows from the properties of $S^n$ that 
$\nu^n({\color{black}\boldsymbol{\theta}})$ \text{is orthogonal to} $T_{\nu^n({\color{black}\boldsymbol{\theta}})}S^n$.
We also define the vector fields $\tilde\mu_1,\dots,\tilde\mu_n\colon W \to TS^n$ along $\nu^n$ given by 
	\[\tilde\mu_i({\color{black}\boldsymbol{\theta}})
	=d\nu^n_{\color{black}\boldsymbol{\theta}}\left(\frac{\partial}{\partial \theta_i}\right)=\frac{\partial \nu^n}{\partial \theta_i}.\]
Note that since $\nu^n$ is a unit vector field,
	\[1=|\nu^n({\color{black}\boldsymbol{\theta}})|^2
	=\nu^n({\color{black}\boldsymbol{\theta}})\cdot 
	\nu^n({\color{black}\boldsymbol{\theta}}) \implies 0=
	2\frac{\partial \nu^n}{\partial \theta_i}({\color{black}\boldsymbol{\theta}})\cdot\nu^n({\color{black}\boldsymbol{\theta}})
	=2\ch{\tilde{\mu_i}}({\color{black}\boldsymbol{\theta}})
	\cdot\nu^n({\color{black}\boldsymbol{\theta}})\]
and $\tilde\mu_i({\color{black}\boldsymbol{\theta}})$ 
\text{are orthogonal to} $\nu^n({\color{black}\boldsymbol{\theta}})$.

\begin{lemma}
	For $1 \leq i < j \leq n$ and ${\color{black}\boldsymbol{\theta}} 
	\in \Reg(\nu)$, 
	$\tilde\mu_i({\color{black}\boldsymbol{\theta}})
	\cdot\tilde \mu_j({\color{black}\boldsymbol{\theta}})=0$.
	Therefore, the set $\{\nu^n({\color{black}\boldsymbol{\theta}}),\tilde 
	\mu_1({\color{black}\boldsymbol{\theta}}),
	\dots,\tilde \mu_n({\color{black}\boldsymbol{\theta}})\}$ 
	is an orthogonal basis for $\mb{R}^{n+1}$.
	Moreover,
	\begin{align*}
		|\tilde \mu_1({\color{black}\boldsymbol{\theta}})|
		=1; && |\tilde \mu_2({\color{black}\boldsymbol{\theta}})|
		=\cos\theta_1; && \dots && |\tilde \mu_n({\color{black}\boldsymbol{\theta}})
		|=
		\cos\theta_1\cdots\cos\theta_{n-1}.
	\end{align*}
\end{lemma}

\begin{proof}
	We first note that the vector fields $\nu^n$ can be written recursively as
		\[\nu^{n+1}({\color{black}\boldsymbol{\theta}},\phi)
		=\nu^n_1({\color{black}\boldsymbol{\theta}})
		(\cos \phi, \sin \phi,0\dots,0)
		+(0,0,\nu^n_2({\color{black}\boldsymbol{\theta}}),
		\dots,\nu^n_n({\color{black}\boldsymbol{\theta}})),\]
	where ${\color{black}\boldsymbol{\theta}} 
	\in \mathbb{R}^n$ and $\phi \in \mathbb{R}$.
	We then proceed by induction over $n$.
	
	For $n=2$, we have
		\[\nu^2(\theta_1,\theta_2)=(\cos\theta_1\cos\theta_2,\cos\theta_1\sin\theta_2,\sin\theta_1),\]
	from which {\color{black}it} follows that
	\begin{align*}
		\tilde\mu_1(\theta_1,\theta_2)&=\frac{\partial \nu^2}{\partial \theta_1}(\theta_1,\theta_2)=(-\sin\theta_1\cos\theta_2,-\sin\theta_1\sin\theta_2,\cos\theta_1); \\
		\tilde\mu_2(\theta_1,\theta_2)&=\frac{\partial \nu^2}{\partial \theta_2}(\theta_1,\theta_2)=(-\cos\theta_1\sin\theta_2,\cos\theta_1\cos\theta_2,0).
	\end{align*}
	We then compute their modules and the dot products among them:
	\begin{align*}
		|\tilde\mu_1|&=\sqrt{\sin^2\theta_1\cos^2\theta_2+\sin^2\theta_1\sin^2\theta_2+\cos^2\theta_1}=\sqrt{\sin^2\theta_1+\cos^2\theta_1}=1;\\
		|\tilde\mu_2|&=\sqrt{\cos^2\theta_1\sin^2\theta_2+\cos^2\theta_1\cos^2\theta_2}=\cos\theta_1;\\
		\tilde\mu_1\cdot\tilde\mu_2&=\sin\theta_1\cos\theta_1\sin\theta_2\cos\theta_2-\sin\theta_1\cos\theta_1\sin\theta_2\cos\theta_2=0.
	\end{align*}
	
	Now assume the statement to be true for some $n > 1$.
	We have that for $i=1,\dots,n$,
	\begin{align*}
		\tilde\mu_i({\color{black}\boldsymbol{\theta}},\phi)&=\frac{\partial \nu^{n+1}}{\partial \theta_i}({\color{black}\boldsymbol{\theta}},\phi)=\frac{\partial \nu^n_1}{\partial \theta_i}({\color{black}\boldsymbol{\theta}})
		(\cos\phi,\sin\phi,0,\dots,0)+\left(0,0,\frac{\partial \nu^n_2}{\partial \theta_i}({\color{black}\boldsymbol{\theta}}),\dots,\frac{\partial \nu^n_n}{\partial \theta_i}({\color{black}\boldsymbol{\theta}})\right);\\
		\tilde\mu_{n+1}({\color{black}\boldsymbol{\theta}},\phi)&
		=\frac{\partial \nu^{n+1}}{\partial \phi}({\color{black}\boldsymbol{\theta}},
		\phi)=\nu^n_1({\color{black}\boldsymbol{\theta}})
		(-\sin\phi,\cos\phi,0,\dots,0).\\
	\end{align*}
	Computing their modules and dot products gives us
	\begin{align*}
		|\tilde\mu_1|&=\sqrt{\left(\frac{\partial \nu^n_1}{\partial \theta_1}\right)^2(\cos^2\phi+\sin^2\phi)+\left(\frac{\partial \nu^n_2}{\partial \theta_1}\right)^2+\dots+\left(\frac{\partial \nu^n_n}{\partial \theta_1}\right)^2}=\left|\frac{\partial \nu^n}{\partial \theta_1}\right|=1;\\
		|\tilde\mu_i|&=\sqrt{\left(\frac{\partial \nu^n_1}{\partial \theta_i}\right)^2(\cos^2\phi+\sin^2\phi)+\left(\frac{\partial \nu^n_2}{\partial \theta_i}\right)^2+\dots+\left(\frac{\partial \nu^n_n}{\partial \theta_i}\right)^2}=\\
			&=\left|\frac{\partial \nu^n}{\partial \theta_i}\right|=\cos\theta_1\dots\cos\theta_{i-1};\\
		|\tilde\mu_{n+1}|&=\sqrt{\nu^n_1({\color{black}\boldsymbol{\theta}})^2(\sin^2\phi+\cos^2\phi)}=\nu^n_1=\cos\theta_1\dots\cos\theta_n;\\
		\tilde\mu_i\cdot\tilde\mu_j&=\frac{\partial \nu^n_1}{\partial \theta_i}\frac{\partial \nu^n_1}{\partial \theta_j}(\cos^2\phi+\sin^2\phi)+\frac{\partial \nu^n_2}{\partial \theta_i}\frac{\partial \nu^n_2}{\partial \theta_j}+\dots+\frac{\partial \nu^n_n}{\partial \theta_i}\frac{\partial \nu^n_n}{\partial \theta_j}=\frac{\partial \nu^n}{\partial \theta_i}\cdot \frac{\partial \nu^n}{\partial \theta_j}=0\\
		\tilde\mu_{n+1}\cdot\tilde\mu_j&=\frac{\partial \nu^n_1}{\partial \theta_j}\nu^n_1(\sin\phi\cos\phi-\cos\phi\sin\phi)=0.
	\end{align*}
	The statement then follows.
\end{proof}

From all this, we get the following

\begin{proposition}\label{nu mu orthonormal basis}
	The set $\{\nu^n({\color{black}\boldsymbol{\theta}}),
	\hat\mu_1({\color{black}\boldsymbol{\theta}}),
	\dots,\hat\mu_n({\color{black}\boldsymbol{\theta}})\}$ 
	is an orthonormal basis for $T_{\nu({\color{black}\boldsymbol{\theta}})}
	\mathbb{R}^{n+1}$, where $\hat\mu_i=\tilde\mu_i/|\tilde\mu_i|$ for $i=1,\dots,n$.
\end{proposition}

\begin{proof}[Proof of Theorem \ref{theorem1}]
	Given ${\color{black}\bf x} \in U_n$, 
	we identify $\mb{R}^{n+1}$ with $T_{f({\color{black}\bf x})}
	\mb{R}^{n+1}$.
	Let $\nu\colon \Reg(f) \to S^n$ be the unit vector field along $f$ and $a=f\cdot \nu$.	
	Since $|\nu({\color{black}\bf x})|=1$ 
	for all ${\color{black}\bf x} \in U_n$, we can choose $\theta_1,\dots,\theta_n\colon \Reg(f) \to W$ such that
		\[\nu=\nu\circ (\theta_1,\dots,\theta_n).\]
	Moreover, $\nu$ verifies the identity $\nu({\color{black}\bf x})
	\cdot df_{\color{black}\bf x}({\color{black}\bf v})=0$ 
	for all ${\color{black}\bf v} \in T_{\color{black}\bf x}\mb{R}^n$ and 
	${\color{black}\bf x} \in \Reg(f)$, hence
		\[\frac{\partial a}{\partial x_i}({\color{black}\bf x})
		=\frac{\partial f}{\partial x_i}({\color{black}\bf x})
		\cdot \nu({\color{black}\bf x})+f({\color{black}\bf x})
		\cdot \frac{\partial \nu}{\partial x_i}({\color{black}\bf x})
		=f({\color{black}\bf x})\cdot 
		\frac{\partial \nu}{\partial x_i}({\color{black}\bf x}).\]
	Using the chain rule, we obtain the following system of differential equations:
	\[
		\frac{\partial a}{\partial x_i}=f\cdot \frac{\partial \nu^n}{\partial x_i}
		=f\cdot \frac{\partial \nu^n}{\partial \theta_1}\frac{\partial \theta_1}{\partial x_i}+\dots+f\cdot 	\frac{\partial \nu^n}{\partial \theta_n}\frac{\partial \theta_n}{\partial x_i}=b_1\frac{\partial \theta_1}{\partial x_i}+\dots+b_n\frac{\partial \theta_n}{\partial x_i},
	\]
	which can be written in matrix form as
	\begin{align}\label{system b1 bn}
		\begin{pmatrix}
			\dfrac{\partial a}{\partial x_1} \\ \vdots \\ \dfrac{\partial a}{\partial x_n}	
		\end{pmatrix}
		=
		\begin{pmatrix}
			\dfrac{\partial \theta_1}{\partial x_1} & \hdots & \dfrac{\partial \theta_n}{\partial x_1}	\\
			\vdots				& \ddots & \vdots				\\
			\dfrac{\partial \theta_1}{\partial x_n} & \hdots & \dfrac{\partial \theta_n}{\partial x_n}
		\end{pmatrix}
		\begin{pmatrix}
			b_1 \\ \vdots \\ b_n
		\end{pmatrix}
	\end{align}

	{\color{black}Notice that} the mapping 
	$d\nu_{\color{black}\bf x}\colon T_{\color{black}\bf x}\Reg(f) 
	\to T_{\nu({\color{black}\bf x})}S^n$ 
	is a monomorphism {\color{black}for any ${\bf x} \in \Reg(\nu)$}.    
	By the chain rule, 
		\[d\nu_{\color{black}\bf x}=
		d\nu^n_{\theta({\color{black}\bf x})}\circ d\theta_{\color{black}\bf x},\]
	hence $d\theta_{\color{black}\bf x}\colon 
	T_{\color{black}\bf x}\Reg(\nu) \to 
	T_{\theta({\color{black}\bf x})}W$ is also a monomorphism.
	Since $\dim T_{\color{black}\bf x}\Reg(\nu) 
	= \dim T_{\theta({\color{black}\bf x})}W$, 
	the Grassman formula implies that $\operatorname{rk} 
	d\theta_{\color{black}\bf x}=\dim T_{\color{black}\bf x}U$, 
	so $d\theta_{\color{black}\bf x}$ is an isomorphism and the coefficient matrix of Equation \eqref{system b1 bn} is invertible at ${\color{black}\bf x}$.
	It follows that the system of equations has a unique solution $b_1,\dots,b_n\colon \Reg(\nu) \to \mb{R}$.

	If we set $\mu_i=\hat\mu_i\circ \theta$ for $i=1,\dots,n$, it follows from Proposition \ref{nu mu orthonormal basis} that we can write 
		\[
		f({\color{black}\bf x})
		=f({\color{black}\bf x})\cdot 
		\nu({\color{black}\bf x})\nu({\color{black}\bf x})
		+f({\color{black}\bf x})\cdot 
		\mu_1({\color{black}\bf x})\mu_1({\color{black}\bf x})
		+\dots+f({\color{black}\bf x})\cdot 
		\mu_n({\color{black}\bf x}) \mu_n({\color{black}\bf x}).
		\] 
	On the one hand, 
	we know that $f({\color{black}\bf x})\cdot 
	\nu({\color{black}\bf x})=
	a({\color{black}\bf x})$ by definition of height function.
	On the other hand, for $j=1,\dots,n$,
		\[f({\color{black}\bf x})\cdot\mu_j({\color{black}\bf x})
		=f({\color{black}\bf x})\cdot \hat 
		\mu_j(\theta({\color{black}\bf x}))
		=f({\color{black}\bf x})\cdot 
		\frac{\tilde \mu_j(\theta({\color{black}\bf x}))}
		{|\tilde \mu_j(\theta({\color{black}\bf x}))|}
		=\frac{f({\color{black}\bf x})}
		{|\tilde \mu_j(\theta({\color{black}\bf x}))|}
		\cdot\frac{\partial \nu^n}{\partial \theta_j}
		(\theta({\color{black}\bf x}))
		=\frac{b_j({\color{black}\bf x})}
		{|\tilde\mu_j(\theta({\color{black}\bf x}))|}.\]
	Therefore,
	\begin{equation}\label{equation to recover f}
		f({\color{black}\bf x})
		=a({\color{black}\bf x})\nu({\color{black}\bf x})
		+\frac{b_1({\color{black}\bf x})}{|\tilde\mu_1(\theta({\color{black}\bf x}))|}
		\mu_1({\color{black}\bf x})+\dots
		+\frac{b_n({\color{black}\bf x})}{|\tilde\mu_n(\theta({\color{black}\bf x}))|}
		\mu_n({\color{black}\bf x})
	\end{equation}
	for any ${\color{black}\bf x} \in \Reg(\nu)$.
	Note that every function on the right-hand side of this identity can be computed entirely in terms of $a$ and $\nu$.
	
	The set $\Reg(\nu)$ is \ch{open and} dense in $\Reg(f)$ by definition of pseudo regular frontal, and since $\Reg(f)$ is \ch{open and} dense in $U_n$, $\Reg(\nu)$ is \ch{open and} dense in $U_n$.
	\text{Since they are continuous}, we can uniquely extend the mapping $\theta\colon \Reg(f) \to W$ and the functions $b_1,\dots,b_n\colon \Reg(\nu) \to \mb{R}$ to $U_n$, which is the domain of $f$.
	We conclude that the identity \eqref{equation to recover f} holds in $U_n$, so it can be written entirely in terms of the Legendre data, as stated.
\end{proof}

\begin{remark}
	Most of the arguments in the proof above can be carried out for any $f\colon U_n \to \mb{R}^{n+1}$ regardless of whether $\Reg(\nu)$ is dense in $\Reg(f)$ or not.
	The only obstruction to this argument lies when one wishes to find the functions $b_1,\dots,b_n$ that verify the equation
		\[da=b_1\,d\theta_1+\dots+b_n\,d\theta_n,\]
	as the solution is only guaranteed to be unique in $\Reg(\nu)$.
	If we denote by $\Sigma^\ell(\nu)$ 
	the set of points ${\color{black}\bf x}\in \Reg(f)$ 
	where $d\nu_{\color{black}\bf x}$ 
	has corank $\ell$, then $d\theta_{\color{black}\bf x}$ 
	has corank $\ell$ and there are indices $1 \leq i_1 < i_2 < \dots < i_\ell \leq n$ such that $b_{i_1}, \dots, b_{i_\ell}\colon \Sigma^\ell(\nu) \to \mb{R}$ can be chosen to be any function.
	 Being able to recover $f$ then depends on being able to choose the appropriate functions $b_{i_1}, \dots, b_{i_\ell}$. In other words, for non pseudo regular frontals there might be infinitely many different mappings sharing the same Legendre data.
	
	There are two ways to solve this conundrum: on the one hand, Theorem 2 claims that the set of pseudo regular frontals $U_n \to \mb{R}^{n+1}$ are dense in ${\color{black}A}(U_n,\mb{R}^{n+1})$, meaning that one can choose a sequence of pseudo regular frontals $f_k\colon U_n \to \mb{R}^{n+1}$ which converge to $f$ pointwise when $k$ goes to infinity (see \S 5 below).
	On the other hand, one can use additional information about $f$, such as the lowest rank $df_x$ attains in $U_n$ (e.g. the cuspidal edge from Example \ref{example1} Item 3 has at least rank $1$, so it can be written in the form
		\[(x,y) \mapsto (x,p(x,y),q(x,y))\]
	by taking suitable coordinates in the source and target).
\end{remark}

\section{Proof of Theorem \ref{theorem2}}\label{section4}
{\color{black} 
\subsection{Proof of the assertion (1) of Theorem \ref{theorem2}} 
\label{subsection4.1}
The proof of Theorem \ref{theorem1} given in Section \ref{section3} 
works well even for pseudo regular 
frontals.   This completes the proof.    
\hfill $\Box$
\subsection{Proof of the assertion (2) of Theorem \ref{theorem2}} 
\label{subsection4.2}   
Let $RM\left(U_n, \mathbb{R}^{n+1}\right)$ be the set 
consisting of {\color{black}real-analytic} mappings $f: U_n\to \mathbb{R}^{n+1}$ 
such that $Reg(f)$ is dense in $U_n$.     The set 
$RM\left(U_n, \mathbb{R}^{n+1}\right)$ is endowed with the Whitney 
$C^\infty$ topology.   
Then,  
$RM\left(U_n, \mathbb{R}^{n+1}\right)$ is a topological subspace 
of $A\left(U_n, \mathbb{R}^{n+1}\right)$ and 
$PRF\left(U_n, \mathbb{R}^{n+1}\right)$ is a topological subspace of 
$RM\left(U_n, \mathbb{R}^{n+1}\right)$.   Namely, the following holds.  
\[
PRF\left(U_n, \mathbb{R}^{n+1}\right)\subset 
RM\left(U_n, \mathbb{R}^{n+1}\right)\subset 
A\left(U_n, \mathbb{R}^{n+1}\right).
\]   
It is clear that the following two lemma\ch{s} 
show the assertion (2) of Theorem \ref{theorem2}.  
\begin{lemma}\label{lemma4.1}
$RM\left(U_n, \mathbb{R}^{n+1}\right)$ is dense in 
 \ch{$A(U_n,\mb{R}^{n+1})$}.  
\end{lemma}
\begin{lemma}\label{lemma4.2}
$PRF\left(U_n, \mathbb{R}^{n+1}\right)$ is dense in 
$RM\left(U_n, \mathbb{R}^{n+1}\right)$.  
\end{lemma}
} 
Before proving {\color{black}Lemma \ref{lemma4.1} and Lemma \ref{lemma4.2},  
several preparations are given.    
\par 
Let $C^\infty\left(U_n, \mathbb{R}^{n+1}\right)$ be the set consisting of 
$C^\infty$ mappings $f: U_n\to \mathbb{R}^{n+1}$.    
The Whitney $C^\infty$ topology is the standard topology 
on $C^\infty(U_n, \mathbb{R}^{n+1})$, 
which is a natural smooth generalisation of the compact open topology 
on the space of continuous mappings 
$U_n \to \mathbb{R}^{n+1}$ (see \cite{Golubitski_Guillemin}, \S 3 for details).
Since any real-analytic mapping $f: U_n\to \mathbb{R}^{n+1}$ is a 
$C^\infty$ mapping, it follows that 
$A(U_n, \mathbb{R}^{n+1})$ is a topological subspace of 
$C^\infty(U_n, \mathbb{R}^{n+1})$.   
\begin{fact}[Identity Theorem]\label{identity_theorem}
Let $p$ be a positive integer,  
$V$ be a connected open subset of $\mathbb{R}^p$ and 
let $F: V\to \mathbb{R}$ be a real-analytic function.   
Suppose that the fiber $F^{-1}(0)$ has an interior point.   
Then, $F({\color{black}\bf x})\equiv 0$ for any ${\color{black}\bf x}\in V$.   
\end{fact}
Fact \ref{identity_theorem} is a beautiful and strong 
fact.   By using Fact \ref{identity_theorem}, Lemma \ref{lemma4.1} and 
Lemma \ref{lemma4.2} can be proved 
in a clear and easy-to-understand manner.    
Notice that the real-analytic assumption for the function $F$ 
cannot be genelarized to the $C^\infty$ assumption.   
For details on this matter, see for instance \S 3 of \cite{brocker}.     
Nevertheless, by using common notions in Singularity Theory of 
Differentiable Mappings such as Jet Spaces, Thom's Jet Transverslity Theorem,  
it is possible to prove Lemma \ref{lemma4.1} 
for $C^\infty$ mappings.    
However, even if notions and results 
in Singularity Theory of Differentiable Mappings found in 
standard textbooks such as \cite{arnoldetall, Golubitski_Guillemin} 
are fully incorporated, 
it seems that the proof of \ch{the}
$C^\infty$ version of Lemma \ref{lemma4.2} 
remains wrapped in mystery.     
The biggest obstruction for the proof of \ch{the} $C^\infty$ version 
of Lemma \ref{lemma4.2} is that 
the source space $Reg(f)$ of \ch{the} Gauss mapping $\nu: Reg(f)\to S^n$ varies 
depending on parametrizations $f: U_n\to \mathbb{R}^{n+1}$.  
{\color{black} 
Then, we pose the following problem here.   
\begin{problem}\label{problem2} 
Does the $C^\infty$ version 
of Lemma \ref{lemma4.2} hold?
\end{problem}
}
{\color{black}On the other hand}, 
even if we restrict ourselves to real-analytic mappings, 
in practical terms, there are almost no problems when applying 
Theorem \ref{theorem1}, Theorem \ref{theorem2} and Theorem \ref{theorem3}.    
Theorefore, in this paper, 
we assume that all functions and mappings are real-analytic 
unless otherwise stated.    
}    
{
\par 
\smallskip 
\begin{proof}[Proof of Lemma \ref{lemma4.1}]   
{\color{black} 
Now, 
we start to give a proof of Lemma \ref{lemma4.1}.   
}
	Let $f\colon U_n \to \R^{n+1}$ be a {\color{black}real-analytic} mapping.   
{\color{black} For any ${\bf x}\in U_n$, set 
\[
{\bf x}=\left(x_1, \ldots, x_n\right).  
\]
Namely, $x_i$ $(1\le i\le n)$ is the $i$-th coordinate of 
${\bf x}\in U_n\subset \mathbb{R}^n=
\underbrace{\mathbb{R}\times \cdots \mathbb{R}}_{n\mbox{-tuples}}$.  
Set $f_j=X_j\circ f$ $(1\le j\le n+1)$ where 
$X_j: \mathbb{R}^{n+1}\to \mathbb{R}$ stands for the 
$j$-th coordinate function of 
$\mathbb{R}^{n+1}=
\underbrace{\mathbb{R}\times \cdots \mathbb{R}}_{(n+1)\mbox{-tuples}}$.   
Thus, we have set  
\[
f({\bf x})=\left(f_1({\bf x}), f_2({\bf x}), \ldots, f_{n+1}({\bf x})\right).   
\]
Consider the Jacobian matrix 
\[
Jf({\bf x})=\left(\frac{\partial f_j}{\partial x_i}({\bf x})\right)
_{1\le i\le n, 1\le j\le n+1}
\] 
of $f$.  
For any $k$ $(1\le k\le n+1)$, let $M_{f, k}({\bf x})$ be the determinant of 
the $n\times n$ matrix given by removing the 
$k$-th \ch{row} ($1 \leq k \leq n+1)$ from $Jf({\bf x})$.   
That is to say, 
\[
M_{f, k}({\bf x})=
\left|
\begin{array}{ccc} 
\frac{\partial f_1}{\partial x_1}({\bf x}) & \cdots 
& \frac{\partial f_1}{\partial x_n}({\bf x}) \\ 
\vdots & \vdots & \vdots \\ 
\widehat{ \frac{\partial f_k}{\partial x_1}({\bf x})} & \widehat{ \cdots} 
& \widehat {\frac{\partial f_k}{\partial x_n}({\bf x})} \\ 
\vdots & \vdots & \vdots \\ 
\frac{\partial f_{n+1}}{\partial x_1}({\bf x}) & \cdots 
& \frac{\partial f_{n+1}}{\partial x_n}({\bf x}) \\ 
\end{array}
\right|, 
\]
where the $\widehat{\cdots}$ denotes deleting $\cdots$.   
Let $\Phi_f: U_n\to \mathbb{R}$ be the real-analytic 
function defined by 
\[
\Phi_f({\bf x})= \sum_{k=1}^{n+1}\ch{(M_{f, k}({\bf x}))^2}.    
\]
}  
	{\color{black}Recall that} a point ${\bf x} \in U_n$ 
	is singular if $df_{\bf x}\colon T_{\bf x}U_n \to T_{f({\bf x})}\R^{n+1}$ is not injective.
	This happens at {\color{black} and only at} points 
	{\color{black}{\bf x}} where 
	{\color{black}all of 
	the $n$-minors $M_{f, k}({\bf x})$ 
	of the Jacobian matrix $Jf({\bf x})$ vanish, which is equivalent to 
	say\ch{ing} that this happens at and only at points ${\bf x}$ satisfying 
	\[
	\Phi_f({\bf x})=0.    
	\]  
	}
\par 
{\color{black}
Suppose that $U_n$ is connected.   Suppose moreover that 
$f\not\in RM\left(U_n, \mathbb{R}^{n+1}\right)$.    The second 
supposition is equivalent to assume that  
the closed set 
\[
Ker(\Phi_{f})=\left\{{\bf x}\in U_n\, |\, \Phi_f({\bf x})=0\right\}
\]
has an interior point.   
Then, by Fact \ref{identity_theorem} (the I\ch{d}entity Theorem) and by 
the first supposition, 
it follows 
\[
\Phi_f({\bf x})\equiv 0 \quad (\mbox{for any }{\bf x}\in U_n).    
\]
Next, we consider linear perturbations of $f$.      
For any $j$ $(1\le j\le n+1)$, let ${\bf c}_{j}\in \mathbb{R}^n$ be an   
$n$-dimensional constant vector.   
\[
{\bf c}_j =\left(c_{(j,1)}, \ldots c_{(j,n)}\right)\in \mathbb{R}^n.  
\]
For the $f\in A(U_n, \mathbb{R}^{n+1})$, let  
$f+{\bf \ch{C}}: U_n\to \mathbb{R}^{n+1}$ be a mapping of the following type.   
\[
\left(f+{\bf \ch{C}}\right)({\bf x})
=\left(f_1({\bf x})+{\bf c}_1\cdot {\bf x}, \ldots, 
f_{n+1}({\bf x})+{\bf c}_{n+1}\cdot {\bf x}\right)
\]
Thus, $f+{\bf \ch{C}}$ is a linear perturbation of $f$.   
Take one point ${\bf x}_0$ of $U_n$ and fix it.    
For any ${\bf \ch{C}}\in \mathbb{R}^{n(n+1)}$, 
set 
\[
\Psi_{f, {\bf x}_0}({\bf \ch{C}})=\Phi_{f+{\bf \ch{C}}}({\bf x}_0).   
\]
If we regard $c_{(j,i)}$ $(1\le i\le n, 1\le j\le n+1)$ as variables, 
then the function $\Psi_{f, {\bf x}_0}: \mathbb{R}^{n(n+1)}\to \mathbb{R}$ 
is a monic polynomial function with degree $2n$, \ch{in particular not identically zero}.   
Hence we have the following.   
\begin{sublemma}\label{sublemma4.1}
The following subset of $\mathbb{R}^{n(n+1)}$
is of Lebesgue measure zero \ch{and closed}.  
\[
Ker\left(\Psi_{f, {\bf x}_0}\right)
=\left\{{\bf \ch{C}}\in \mathbb{R}^{n(n+1)}\, 
|\, \Psi_{f, {\bf x}_0}({\bf \ch{C}})=0\right\}.   
\leqno{(4.1)}
\]
\end{sublemma} 
By Fact \ref{identity_theorem} and Sublemma \ref{sublemma4.1}, 
for any ${\bf \ch{C}}\in \mathbb{R}^{n(n+1)}$ satifying 
$\Psi_{f, {\bf x}_0}({\bf \ch{C}})\ne 0$, the linear perturbation 
$f+{\bf \ch{C}}$ must be contained 
in $RM\left(U_n, \mathbb{R}^{n+1}\right)$.   
Moreover, again by Sublemma \ref{sublemma4.1}, 
for any positive integer $\ell$ \ch{there exists ${\bf \ch{C}}\in \mathbb{R}^{n(n+1)}$ such that} the linear perturbation 
$f+{\bf \ch{C}}$ \ch{satisfies} 
the following condition.   
\[
\sum_{j=1}^{n+1}\left(\sum_{i=1}^n c_{(j,i)}^2\right) < \frac{1}{\ell^2}.   
\leqno{(4.2)}
\]  
For any positive integer $\ell$, take one constant vector 
${\bf \ch{C}}_0[\ell]\in \mathbb{R}^{n(n+1)}$ satisfying (4.1) and (4.2) and fix it.   
Then, it is clear that 
\[
f+{\bf \ch{C}}_0[\ell] \in RM\left(U_n, \mathbb{R}^{n+1}\right) 
\quad (\mbox{for any }\ell\in \mathbb{N}) 
\leqno{(4.3)}
\]
and 
\[
\lim_{\ell\to \infty}\left(f+{\bf \ch{C}}_0[\ell]\right)=f.   
\leqno{(4.4)}
\]
\par 
\smallskip 
Suppose that $U_n$ is not connected.    
Then, since $\mathbb{R}^n$ satisfies the axiom of second countability, 
the number of connected components of $U_n$ is countable.   
Since the union of countably many 
Lebesgue measure zero subsets 
in $\mathbb{R}^{n(n+1)}$  
is a subset of Lebesgue measure zero, we may again choose 
one constant vector 
${\bf c}_0[\ell]\in \mathbb{R}^{n(n+1)}$ satisfying (4.1) and (4.2) for any 
$\ell\in \mathbb{N}$.      
Thus, we can again obtain (4.3) and (4.4).  
This completes the proof.   
}
\end{proof}   
\bigskip 
{\color{black} 
\begin{proof}[Proof of Lemma \ref{lemma4.2}]         
{\color{black} 
Suppose that $U_n$ is connected.   
Let $f$ be a mapping contained in $RM(U_n, \mathbb{R}^{n+1})$.
\ch{There is a $k\in\{1,\ldots,n+1\}$ such that $M_{f,k}^{-1}(0)$ has no interior points.
If this was not the case, we would have by Fact \ref{identity_theorem} that $M_{f,1},\dots,M_{f,n+1}$ all vanish on $U_n$, meaning that
	\[\Phi_f({\bf x})= \sum_{k=1}^{n+1}\ch{(M_{f, k}({\bf x}))^2}\equiv 0\]
and $Reg(f)=U_n\backslash \ker(\Phi_f)=\emptyset$, hence $f$ would not be in $RM(U_n,\mathbb{R}^{n+1})$.
We can thus assume,}
by permutating coordinates $\left(X_1, \ldots, X_{n+1}\right)$ 
of $\mathbb{R}^{n+1}$ in advance if necessary, that
\ch{$M_{f,{n+1}}^{-1}(0)$ has no interior points, hence}
$Reg(F)$ is 
dense in $U_n$,      
where the mapping $F: U_n\to \mathbb{R}^n$ is defined by 
\[
F({\bf x})= \left(X_1\circ f({\bf x}), \ldots, X_n\circ f({\bf x}))\right).   
\]
We \ch{are going to} construct a sequence 
$\left\{f_\ell: U_n\to \mathbb{R}^{n+1}\right\}_{\ell\in \mathbb{N}}$ 
such that the following three are satisfied.    
\smallskip 
\begin{enumerate}
\item[$(4.A)$] For any $\ell\in \mathbb{N}$, 
$Reg\left(f_\ell\right)$ is 
dense in $U_n$.       
\item[$(4.B)$] For any $\ell\in \mathbb{N}$, 
$Reg\left(\nu_{f_\ell}: Reg\left(f_\ell\right)\to S^n\right)$ is 
dense in $Reg\left(f_\ell\right)$.    
\item[$(4.C)$] $\lim_{\ell\to \infty}f_\ell=f$.   
\end{enumerate}
\par 
\smallskip 
Set 
\[
X_i\circ f_\ell =X_i\circ f=X_i\circ F\quad (1\le i\le n).   
\]
Then, it is clear that 
\[
Reg\left(f_\ell\right)\supset Reg(F).    
\]
Hence, from the assum\ch{p}tion above, $(4.A)$ follows.   
From now on, for any 
$\ell\in \mathbb{N}$, 
we construct $X_{n+1}\circ f_\ell$ satifying  
$(4.B)$ and $(4.C)$.  
For any sufficiently small positive number $\varepsilon$ and 
any ${\bf x}_0\in U_n$ let $D_{\left({\bf x}_0, \varepsilon\right)}$ be the 
open disc centered at ${\bf x}_0$ with radius $\varepsilon$.   Namely, 
\[
D\left({\bf x}_0, \varepsilon\right)
=\left\{\left. {\bf x}\in \mathbb{R}^n\; \right|\; 
\left({\bf x}-{\bf x}_0\right)\cdot \left({\bf x}-{\bf x}_0\right)
< \varepsilon^2\right\}. 
\]        
For any  ${\bf x}\in D\left({\bf x}_0, \varepsilon\right)$, 
define the $n$-dimensional vector 
$F^2({\bf x})\in \mathbb{R}^n$ by 
\[
F^2({\bf x})=\left(\left(X_1\circ F({\bf x}\right) - 
\left.X_1\circ F({\bf x}_0)\right)^2, \ldots, 
\left(X_n\circ F({\bf x}\right) - 
\left.X_n\circ F({\bf x}_0)\right)^2\right).   
\]
Then define the function $X_{n+1}\circ f_\ell: U_n\to \mathbb{R}$ by 
\[
X_{n+1}\circ f_\ell({\bf x})= X_{n+1}\circ f({\bf x}) + 
\frac{1}{2}\left({\bf c}_\ell\cdot F^2({\bf x})\right), 
\]
where ${\bf c}_\ell =\left(c_{(\ell, 1)}, \ldots, c_{(\ell, n)}\right)$ is an 
$n$-dimensional constant vector in $\mathbb{R}^n$.
By Fact \ref{identity_theorem}, $(4.B)$ is equivalent to the following 
$(4.B.1)$.    
\smallskip 
\begin{enumerate}
\item[$(4.B.1)$] For any $\ell\in \mathbb{N}$, any sufficiently small 
positive number $\varepsilon>0$ and any ${\bf x}_0\in Reg(F)$ satisfying 
$D\left({\bf x}_0, \varepsilon\right)\subset Reg(F)$, 
the set 
\[
Reg\left(\nu_{f_\ell}|_{D\left({\bf x}_0, \varepsilon\right)}
: D\left({\bf x}_0, \varepsilon\right)\to S^n\right)\ch{=Reg\left(\nu_{f_\ell}\right) \cap D({\bf x}_0,\varepsilon)}
\] 
is 
dense in $D\left({\bf x}_0, \varepsilon\right)$.   
\end{enumerate}
\par  
\smallskip 
\ch{
	The argument is as follows: if $Reg(\nu_{f_\ell})$ is not 
	dense in $Reg(f_\ell)$, then $\ker(\Phi_{\nu_{f_\ell}})$ must have an interior point and $Reg(\nu_{f_\ell})$ is empty by Fact \ref{identity_theorem}, thus $Reg\left(\nu_{f_\ell}|_{D\left({\bf x}_0, \varepsilon\right)}\right)$ is empty for all ${\bf x}_0 \in Reg(f_\ell)$.
Conversely, if there is an ${\bf x}_0 \in Reg(f)$ such that $Reg\left(\nu_{f_\ell}|_{D\left({\bf x}_0,\varepsilon\right)}\right)$ is not 
dense in $D({\bf x}_0,\varepsilon)$, then $\ker(\Phi_{\nu_{f_\ell}})$ has an interior point in $D({\bf x}_0,\varepsilon)$ and $Reg\left(\nu_{f_\ell}|_{D\left({\bf x}_0, \varepsilon\right)}\right)$ is empty by Fact \ref{identity_theorem}, thus $Reg\left(\nu_{f_\ell}\right)$ is not 
dense in $Reg(f_\ell)$.
}

For any $i$ $(1\le i\le n)$, let 
$\widetilde{x}_i: D\left({\bf x}_0, \varepsilon\right)\to \mathbb{R}$ 
be the function 
defined by 
\[
\widetilde{x}_i=X_i\circ 
\left(F|_{D\left({\bf x}_0, \varepsilon\right)}\right).   
\] 
Since $F|_{D({\bf x}_0, \varepsilon)}: 
D({\bf x}_0, \varepsilon)
\to F\left(D({\bf x}_0, \varepsilon)\right)$ 
is an analytic diffeomorphism, 
$\left(D({\bf x}_0, \varepsilon), 
\left(\widetilde{x}_1, \ldots, \widetilde{x}_n\right)\right)$ is regarded 
as a coordinate neighborhood of $\mathbb{R}^n$.    
Set $\widetilde{\bf x}=\left(\widetilde{x}_1, \ldots, \widetilde{x}_n\right)$.   
Then, the mapping $f_\ell{|}_{D_{\left({\bf x}_0, \varepsilon\right)}}: 
\ch{D\left({\bf x}_0, \varepsilon\right)\to \mathbb R^{n+1}}$ 
may be described with respect to coordinates 
$\left(\widetilde{x}_1, \ldots, \widetilde{x}_n\right)$ as follows.   
\begin{eqnarray*}
X_i\circ f_\ell\left(\widetilde{\bf x}\right) 
& = & \widetilde{x}_i\quad (1\le i\le n),  \\  
X_{n+1}\circ f_\ell \left(\widetilde{\bf x}\right) 
& = & 
X_{n+1}\circ f(\widetilde{\bf x}) + 
\frac{1}{2}\left({\bf c}_\ell\cdot F^2(\widetilde{\bf x})\right)  \\ 
{ } & = & 
X_{n+1}\circ f(\widetilde{\bf x}) + 
\frac{1}{2}\sum_{i=1}^nc_{(\ell, i)}
\left(\widetilde{x}_i-\widetilde{x}_{(0, i)}\right)^2,      
\end{eqnarray*}
where $\widetilde{\bf x}\left({\bf x}_0\right)=
\left(\widetilde{x}_{(0, 1)}, \ldots, \widetilde{x}_{(0, n)}\right)$.   
For any ${\bf x}_0\in \mathbb{R}^n$, 
let $h: \mathbb{R}^n\to \mathbb{R}^n$ be the parallel translation 
in $\mathbb{R}^n$ defined by 
\[
h({\bf x})={\bf x}+{\bf x}_0.    
\]
For any $\ell\in \mathbb{N}$, any sufficiently small positive 
$\varepsilon >0$ and any ${\bf x}_0\in Reg(F)$, let 
$H_\ell: \mathbb{R}^{n+1} \to 
\mathbb{R}^{n+1}$ be the affine transformation of $\mathbb{R}^{n+1}$ 
defined by 
\begin{eqnarray*}
{ } & { } & H_\ell\left(X_1, \ldots, X_n, X_{n+1}\right)  \\ 
{ } & = &  
\left(X_1-\widetilde{x}_{(0, 1)}, \ldots,\; X_n- \widetilde{x}_{(0,n)}, 
\; X_{n+1}- X_{n+1}\circ 
f\left(\widetilde{\bf x}\left({\bf x}_0\right)\right)  
- \sum_{i=1}^n\frac{\partial \left(X_{n+1}\circ f\right)}{\partial \widetilde{x}_i}
\left(\widetilde{\bf x}\left({\bf x}_0\right)\right)
\left(X_i- \widetilde{x}_{(0,i)}\right)\right).     
\end{eqnarray*}
For any sufficiently small $\varepsilon>0$, any ${\bf x}_0\in Reg(F)$  
and any $\ell\in \mathbb{N}$, 
define the mapping 
$
\widetilde{f}_\ell: D\left({\bf 0}, \varepsilon\right)\to \mathbb{R}^{n+1} 
$
by 
\[
\widetilde{f}_{\ell}=H_\ell\circ f_{\ell}\circ 
\left(h|_{D\left({\bf 0}, \varepsilon\right)}\right). 
\]     
Then, the convergent power series of 
$X_i\circ \widetilde{f}_{\ell}$ $(1\le i\le n)$ around the origin of 
$\mathbb{R}^n$ has the following form where 
$\widetilde{y}_i=\widetilde{x}_i-\widetilde{x}_{(0,i)}$.   
\begin{eqnarray*}
{ } & { } &  X_i\circ \widetilde{f}_\ell
\left(\widetilde{y}_1, \ldots, \widetilde{y}_n\right)  \\
{ } & = & 
X_i\circ H_\ell\circ f_\ell\circ h\left(\widetilde{y}_1, \ldots, \widetilde{y}_n\right) 
\\ 
{ } & = &  
X_i\circ H_\ell\circ f_\ell 
\left(\widetilde{y}_1+\widetilde{x}_{(0,1)}, \ldots, 
\widetilde{y}_n+\widetilde{x}_{(0,n)}
\right) \\ 
{ } & = & 
X_i\circ H_\ell 
\left(\widetilde{y}_1+\widetilde{x}_{(0,1)}, \ldots, 
\widetilde{y}_n+\widetilde{x}_{(0,n)}, 
X_{n+1}\circ f_\ell 
\left(\widetilde{y}_1+\widetilde{x}_{(0,1)}, \ldots, 
\widetilde{y}_n+\widetilde{x}_{(0,n)}
\right) 
\right) \\ 
{ } & = & 
X_i\left(\widetilde{y}_1, \ldots, \widetilde{y}_n , 
X_{n+1}\circ H_\ell \circ f_\ell 
\left(\widetilde{y}_1+\widetilde{x}_{(0,1)}, \ldots, 
\widetilde{y}_n+\widetilde{x}_{(0,n)} 
\right) 
\right) \\ 
{ } & = &  \widetilde{y}_i.   
\end{eqnarray*}
On the other hand, from the construction of the affine transformation 
$H_\ell$, it is easily seen that the following holds.
\begin{sublemma}\label{sublemma_powerseries}
For any $\ell\in \mathbb{N}$, 
the convergent power series of $X_{\ch{n+1}}\circ \widetilde{f}_\ell$
around the origin of $\mathbb{R}^n$ starts from the quadratic terms.   
Namely, the following two holds for any $\ell\in \mathbb{N}$.    
\begin{enumerate}
\item[(1)] $X_{n+1}\circ \widetilde{f}_\ell({\bf 0})=0$.   
\item[(2)] 
For any $i$ $(1\le i\le n)$, we have the following.   
\[
\frac{\partial \left(X_{n+1}\circ \widetilde{f}_\ell\right)}
{\partial \widetilde{y}_i}({\bf 0}) = 0.   
\]
\end{enumerate}
\end{sublemma}
Notice that $\widetilde{f}_\ell\left({\bf 0}\right)={\bf 0}$ 
and that by Sublemma \ref{sublemma_powerseries}, 
the power series expression of $X_{n+1}\circ \widetilde{f}_\ell$ 
around the origin has no linear terms for any $\ell\in \mathbb{N}$.    
Notice moreover that \ch{since $h$ is a homeomorphism}, $(4.B.1)$ is equivalent to the following $(4.B.2)$
\smallskip 
\begin{enumerate}
\item[$(4.B.2)$] For any $\ell\in \mathbb{N}$, any sufficiently small 
positive number $\varepsilon>0$ and any ${\bf x}_0\in Reg(F)$ satisfying 
$D\left({\bf x}_0, \varepsilon\right)\in Reg(F)$, 
the set 
\[
Reg\left(\nu_{\widetilde{f}_\ell}
: D\left({\bf 0}, \varepsilon\right)\to S^n\right)
\] 
is 
dense in $D\left({\bf 0}, \varepsilon\right)$.  
\end{enumerate}
\smallskip 
The Jacobian matrix of $\widetilde{f}_\ell$ with respect to 
partial derivatives of 
$X_j\circ \widetilde{f}_\ell$ $(1\le j\le n+1)$ 
by $\widetilde{y}_i$ $(1\le i\le n)$ is as follows.   
\[
\ch{\left(
\begin{array}{ccccc}
1 		& 0 		& \cdots 	& 0 \\ 
0 		& 1 		& \cdots 	& 0 \\ 
\vdots 	& \vdots & \ddots & \vdots \\ 
0 		& 0 		& \cdots 	& 1 \\ 
\frac{\partial \left(X_{n+1}\circ \widetilde{f}_\ell\right)}
{\partial \widetilde{y}_1}
(\widetilde{\bf y}) & 
\frac{\partial \left(X_{n+1}\circ \widetilde{f}_\ell\right)}
{\partial \widetilde{y}_2}
(\widetilde{\bf y}) & \cdots & 
\frac{\partial \left(X_{n+1}\circ \widetilde{f}_\ell\right)}
{\partial \widetilde{y}_n}
(\widetilde{\bf y})
\end{array}
\right).}
\]
Thus, the Gauss mapping 
$\nu_{\widetilde{f}_\ell}: D({\bf 0}, \varepsilon)\to S^n$ 
must be one of the following two.    
\begin{eqnarray*}
D\left({\bf 0}, \varepsilon\right)\ni \widetilde{\bf y} & \mapsto &  
\frac{1}{\sqrt{\sum_{i=1}^n 
\left(\frac{\partial \left(X_{n+1}\circ\widetilde{f}_{\ell}\right)}
{\partial \widetilde{y}_i}
\left(\widetilde{\bf y}\right)\right)^2+1}}
\left(
-\frac{\partial \left(X_{n+1}\circ\widetilde{f}_{\ell}\right)}
{\partial \widetilde{y}_1}
\left(\widetilde{\bf y}\right), \cdots, 
-\frac{\partial \left(X_{n+1}\circ \widetilde{f}_{\ell}\right)}
{\partial \widetilde{y}_n}
\left(\widetilde{\bf y}\right), 1\right)\in S^n \\ 
D\left({\bf 0}, \varepsilon\right)\ni \widetilde{\bf y} & \mapsto &  
\frac{1}{\sqrt{\sum_{i=1}^n 
\left(\frac{\partial \left(X_{n+1}\circ\widetilde{f}_{\ell}\right)}
{\partial \widetilde{y}_i}
\left(\widetilde{\bf y}\right)\right)^2+1}}
\left(
\frac{\partial \left(X_{n+1}\circ\widetilde{f}_{\ell}\right)}
{\partial \widetilde{y}_1}
\left(\widetilde{\bf y}\right), \cdots, 
\frac{\partial \left(X_{n+1}\circ\widetilde{f}_{\ell}\right)}
{\partial \widetilde{y}_n}
\left(\widetilde{\bf y}\right), -1\right)\in S^n.   
\end{eqnarray*}
\ch{These expressions can be obtained by computing the cross product of the columns of the Jacobian matrix above.}
Without loss of generality, we may assume 
\[
\nu_{\widetilde{f}_{\ell}}\left(\widetilde{\bf y}\right) = 
\frac{1}{\sqrt{\sum_{i=1}^n 
\left(\frac{\partial \left(X_{n+1}\circ \widetilde{f}_{\ell}\right)}
{\partial \widetilde{y}_i}
\left(\widetilde{\bf y}\right)\right)^2+1}}
\left(
-\frac{\partial \left(X_{n+1}\circ \widetilde{f}_{\ell}\right)}
{\partial \widetilde{y}_1}
\left(\widetilde{\bf y}\right), \cdots, 
-\frac{\partial \left(X_{n+1}\circ \widetilde{f}_{\ell}\right)}
{\partial \widetilde{y}_n}
\left(\widetilde{\bf y}\right), 1\right).     
\]
Since $X_{n+1}\circ \widetilde{f}_\ell$ does not have linear terms 
for any $\ell\in \mathbb{N}$, it follows 
\[
\nu_{\widetilde{f}_\ell}\left({\bf 0}\right) = 
\left(0, \ldots, 0, 1\right)\in S^n.     
\]
Set ${\bf e}_{n+1}=(0, \ldots, 0, 1)\in S^n$.    
Since $\mathbb{R}^{n+1}$ is a real affine space 
and the $n$-dimensional unit sphere 
$S^n$ is canonically embedded in $\mathbb{R}^{n+1}$, 
$T_{{\bf e}_{n+1}}S^n$ is naturally identified with $\mathbb{R}^n\times \{1\}$.    
Under this identification, for any $\ell\in \mathbb{N}$ define the 
mapping $N_{\widetilde{f}_\ell}: D\left({\bf 0}, \varepsilon\right)
\to T_{{\bf e}_{n+1}}S^n=\mathbb{R}^n\times \{1\}$ 
by 
\[
N_{\widetilde{f}_\ell}(\widetilde{\bf y}) = 
\left(
-\frac{\partial \left(X_{n+1}\circ \widetilde{f}_{\ell}\right)}
{\partial \widetilde{y}_1}
\left(\widetilde{\bf y}\right), \cdots, 
-\frac{\partial \left(X_{n+1}\circ \widetilde{f}_{\ell}\right)}
{\partial \widetilde{y}_n}
\left(\widetilde{\bf y}\right), 1\right).    
\]  
Then, 
since any normal coordinates 
around the point ${\bf e}_{n+1}$   
is nothing but 
a local inverse of the exponential mapping from $T_{{\bf e}_{n+1}}S^n$ to 
$S^n$, 
we have the following.    
\begin{sublemma}\label{sublemma4.2}
Let $(V, \xi)$ be a normal coordinate neighborhood of 
$S^n$ at ${\bf e}_{n+1}$.   Then, there exists a real-analytic 
diffemorphism $\varphi: T_{{\bf e}_{n+1}}S^n \to T_{{\bf e}_{n+1}}S^n$ 
such that 
\[
\varphi\circ \xi\circ \nu_{\widetilde{f}_\ell}\left(\widetilde{\bf y}\right) = 
N_{\widetilde{f}_\ell}\left(\widetilde{\bf y}\right)     
\]
for any $\ell\in \mathbb{N}$ and any $\widetilde{\bf y}\in 
D\left({\bf 0}, \varepsilon\right)$.     
\end{sublemma}
By Sublemma \ref{sublemma4.2}, it follows that 
$(4.B.2)$ is equivalent to the following $(4.B.3)$.   
\smallskip 
\begin{enumerate}
\item[$(4.B.3)$] For any $\ell\in \mathbb{N}$, any sufficiently small 
positive number $\varepsilon>0$ and any ${\bf x}_0\in Reg(F)$ satisfying 
$D\left({\bf x}_0, \varepsilon\right)\subset Reg(F)$, 
there exists a sufficiently small positive number $\delta$ 
$(0< \delta \le \epsilon)$ such that 
the set 
\[
Reg\left(N_{\widetilde{f}_\ell}|_{D\left({\bf 0}, \delta\right)}
: D\left({\bf 0}, \delta\right)\to T_{{\bf e}_{n+1}}S^n\right)
\] 
is 
dense in $D\left({\bf 0}, \delta\right)$.  
\end{enumerate}
\smallskip 
\par 
\smallskip 
Recall that $X_{n+1}\circ f_\ell$ is a quadratic perturbation of 
$X_{n+1}\circ f$ having the following form and also that $h$ is merely 
a parallel transformation of $\mathbb{R}^n$.    
\[
X_{n+1}\circ f_\ell ({\bf x}) = 
X_{n+1}\circ f({\bf x}) + 
\frac{1}{2}\left({\bf c}_\ell\cdot F^2({\bf x})\right).   
\]
Hence the second derivative of $X_{n+1}\circ f_\ell$ with respect to 
$\left(\widetilde{y}_1, \ldots, \widetilde{y}_n\right)$ has the following form.   
\[
\frac{\partial^2 \left(X_{n+1}\circ \widetilde{f}_\ell\right)}
{\partial \widetilde{y}_i\partial \widetilde{y}_j}(\widetilde{\bf y}) = 
\left\{
\begin{array}{lc}
\frac{\partial^2 \left(X_{n+1}\circ f\right)}
{\partial \widetilde{x}_i^2}(\widetilde{\bf x}) + c_{(\ell, i)} 
& (\mbox{if }i=j) \\ 
\frac{\partial^2 \left(X_{n+1}\circ f\right)}
{\partial \widetilde{x}_i\partial \widetilde{x}_j}(\widetilde{\bf x})
& (\mbox{if }i\ne j).   
\end{array}
\right.
\leqno{(4.5)}
\]
Denote by $\Psi_{\widetilde{f}_\ell, {\bf 0}}\left({\bf c}_\ell\right)$ 
the Jacobian determinant of 
$N_{\widetilde{f}_\ell}|_{D\left({\bf 0}, \delta\right)}$ \ch{in $S^n$} at ${\bf 0}$.    
By (4.5), if we regard $c_{(\ell, i)}$ $(1\le i\le n)$ as variables, 
then the Jacobian determinant of 
$N_{\widetilde{f}_\ell}|_{D\left({\bf 0}, \delta\right)}$ at ${\bf 0}$ 
is a monic polynomial 
with degree $n$.    
Thus, by Fact \ref{identity_theorem}, we have the following sublemma.   
\begin{sublemma}\label{sublemma4.3}
For any $\ell\in \mathbb{N}$ and any ${\bf x}_0\in Reg(F)$, 
the following subset of 
$\mathbb{R}^n$ is of Lebesgue measure zero \ch{and closed}.   
\[
Ker\left(\Psi_{\widetilde{f}_\ell, {\bf 0}}\right) = 
\left\{\left.
{\bf c}_{\ell}\in \mathbb{R}^n\, \right|\, 
\Psi_{\widetilde{f}_\ell, {\bf 0}}\left({\bf c}_\ell\right)=0
\right\}.   
\]
\end{sublemma}
By 
Sublemma \ref{sublemma4.3}, 
for any $\ell\in \mathbb{N}$, any ${\bf x}_0$ contained in $Reg(F)$ 
and almost all ${\bf c}_{\ell}\in \mathbb{R}^n$ in the sense of 
Lebesgue measure, 
$(4.B.3)$ is satisfied.   
Moreover, again by Sublemma \ref{sublemma4.3}, for any $\ell\in \mathbb{N}$ 
and almost all ${\bf c}_{\ell}\in \mathbb{R}^n$ (in the sense of 
Lebesgue measure) satisfying    
\[
\sum_{i=1}^n c_{(\ell, i)}\ch{^2} < \frac{1}{\ell^2},     
\leqno{(4.6)}
\] 
${\bf c}_{\ell}\in \mathbb{R}^n$ does not belong to 
$Ker\left(\Psi_{\widetilde{f}_\ell, {\bf 0}}\right)$.   
The inequality $(4.6)$ implies 
that we can construct a sequence of quadratic perturbations
$\left\{f_\ell: U_n\to \mathbb{R}^{n+1}\right\}_{\ell\in \mathbb{N}}$ 
satisfying even $(4.C)$ as well.    
\par 
\smallskip 
Suppose that $U_n$ is not connected.    
Then, for each connected component of $U_n$, the proof in the case that 
$U_n$ is connected works well.    
Since $\mathbb{R}^n$ satisfies the axiom of second countability, 
the number of connected components of $U_n$ is countable,    
Since the union of countably many 
Lebesgue measure zero subsets 
in $\mathbb{R}^n$ 
is a subset of Lebesgue measure zero, 
we may again construct a sequence of 
quadratic perturbations 
$\left\{f_\ell: U_n\to \mathbb{R}^{n+1}\right\}_{\ell\in \mathbb{N}}$ 
satisfying $(4.A)$, $(4.B)$ and $(4.C)$.     
This completes the proof.   
}
\end{proof}
\section{Proof of Theorem 3}\label{section5}
Take any $f$ of $A\left(U_n, \mathbb{R}^{n+1}\right)$ and fix it.   
By the assertion (1) of Theorem \ref{theorem2}, there exists a 
sequence $\left\{f_i\right\}_{i=1, 2, \ldots}
\subset PRF\left(U_n, \mathbb{R}^{n+1}\right)$ satisfying 
$\lim_{i\to \infty}f_i=f$.     Take any point ${\color{black}\bf x}$ of $U_n$ and fix it.   
Since $f_i\in PRF\left(U_n, \mathbb{R}^{n+1}\right)$ for any 
$i$ $(i=1, 2, \ldots)$, 
by the assertion (1) of Theorem \ref{theorem2}, 
there must exist a sequence of points 
$\left\{\color{black}{\bf x}_j\right\}_{j=1, 2, \ldots}\subset U_n$ satisfying 
$\lim_{j\to \infty}{\color{black}\bf x}_j=x$, 
$\left\{{\color{black}\bf x}_j\right\}_{j=1, 2, \ldots}\subset Reg\left(f_i\right)$ for any 
$i$ $(i=1, 2, \ldots)$ and $f_i({\color{black}\bf x})$ is completely recovered 
from the sequence of Legendre data 
$\left\{\left\{\nu_i\left({\bf x}_j\right), a_i\left({\bf x}_j\right)
\right\}_{{\color{black}\bf x}_j\in Reg\left(\nu_i\right)}\right\}_{i=1, 2, \ldots}$ 
of $f_i$.    
On the other hand, $f({\color{black}\bf x})
=\lim_{i\to \infty}f_i({\color{black}\bf x})$.    Therefore, 
$f({\color{black}\bf x})$ is  completely recovered 
from the sequence of Legendre data 
$\left\{\left\{\nu_i\left({\color{black}\bf x}_j\right), 
a_i\left({\color{black}\bf x}_j\right)  
\right\}_{{\color{black}\bf x}_j\in Reg\left(\nu_i\right)}\right\}_{i=1, 2, \ldots}$ 
of $f_i$.      
This completes the proof.   
\hfill $\Box$
}
\section{Examples}\label{examples}

\begin{example}[Constant mapping]
	Let ${\bf X}_0 \in \R^{n+1}$ 
	{be a fixed point}, and consider the constant mapping 
	$f\colon \R^n \to \R^{n+1}$ given by $f({\bf x})
	={\bf X_0}$.
	Then $f$ is {a frontal but not a regular frontal}.

	Now consider the sequence of concentric spheres $f_k\colon \R^n \to \R^{n+1}$ given by
		\[f_k({\bf x})=\ch{{\bf X}_0}+\frac{1}{k}\left(\cos x_1 \cdots \cos x_n, \cos x_1 \cdots \cos x_{n-1} \sin x_n,  \cos x_1 \cdots \cos x_{n-2} \sin x_{n-1} \dots, \sin x_1\right),\]
	We see that $\Reg(f_{k})=\R^n$ and $f_k(\R^n)$ is the sphere of centre 
	$\ch{{\bf X}_0}$ and radius $1/k$, 
	hence its Gauss map $\nu_{k}\colon \R^n \to S^n$ is given by
		\[\nu({\bf x})=k(f_k({\bf x})-{\bf X}_0).\]
	Therefore, $\Reg(\nu_{k})=\R^n$ and $f_k$ is a regular frontal for any $k \in \mathbb{N}$.
	We also have that
		\[a_k({\bf x})=
		f_k({\bf x})\cdot \nu({\bf x})
		=\ch{{\bf X}_0}\cdot \nu({\bf x})+\frac{1}{k}.\]
	
	It is clear by the definition of $\nu_{k}$ that 
	$\theta_j(x)=x_j$ for $j=1,\dots,n$, meaning that
		\[b_j=b_j\frac{\partial \theta_j}{\partial x_j}=b_1 \frac{\partial \theta_1}{\partial x_j}+\cdots+b_n \frac{\partial \theta_n}{\partial x_j}=\frac{\partial a_k}{\partial x_j}=x_0\cdot \frac{\partial \nu}{\partial x_j}=x_0\cdot \mu_j,\]
	and from here we obtain
		\[a_k\nu+b_1\mu_1+\cdots+b_n\mu_n=(x_0\cdot \nu)\nu+(x_0\cdot \mu_1) \mu_1+\cdots+(x_0\cdot \mu_n) \mu_n+\frac{1}{k}\nu=x_0+\frac{1}{k}[k(f_k-x_0)]=f_k.\]
\end{example}

\begin{example}[Cuspidal cross-cap]
Let $f\colon \mb{R}^2 \to \mb{R}^3$ be the folded Whitney umbrella,
	\[f(x,y)=(x,y^2,xy^3).\]
This is a regular frontal with $\Reg(f)=\Reg(\nu)=\{(x,y)\in \mb{R}^2: y \neq 0\}$.
The Legendre data associated to $f$ is
\begin{align*}
	\nu(x,y)=\frac{(-2 y^3, -3xy, 2)}{\sqrt{9 x^2 y^2+4 y^6+4}}; && a(x,y)=-\frac{3 x y^3}{\sqrt{9 x^2 y^2+4 y^6+4}}.
\end{align*}
Writing $\nu(x,y)=\nu^2(\theta_1(x,y),\theta_2(x,y))$ gives us functions
\begin{align*}
	\sin \theta_1(x,y)=&\frac{2}{\sqrt{9 x^2 y^2+4 y^6+4}}; 	&\cos \theta_1(x,y)=&\sqrt{\frac{9 x^2 y^2+4 y^6}{9 x^2 y^2+4 y^6+4}};\\
	\sin \theta_2(x,y)=&-\frac{3 x y}{\sqrt{9 x^2 y^2+4 y^6}}; 	&\cos \theta_2(x,y)=&-\frac{2 y^3}{\sqrt{9 x^2 y^2+4 y^6}}.
\end{align*}
Applying the chain rule, we see that
\begin{align*}
	\frac{\partial \theta_1}{\partial x}=\frac{1}{\cos \theta_1}\frac{\partial \sin \theta_1}{\partial x}&=-\frac{18 x y^2}{\sqrt{9 x^2 y^2+4 y^6} \left(9 x^2 y^2+4 y^6+4\right)};
	&
	\frac{\partial \theta_2}{\partial x}=\frac{1}{\cos \theta_2}\frac{\partial \sin \theta_2}{\partial x}&=\frac{6 y^2}{9 x^2+4 y^4};
	\\
	\frac{\partial \theta_1}{\partial y}=\frac{1}{\cos \theta_1}\frac{\partial \sin \theta_1}{\partial y}&=-\frac{6 \left(3 x^2 y+4 y^5\right)}{\sqrt{9 x^2 y^2+4 y^6} \left(9 x^2 y^2+4 y^6+4\right)};&
	\frac{\partial \theta_2}{\partial y}=\frac{1}{\cos \theta_2}\frac{\partial \sin \theta_2}{\partial y}&=-\frac{12 x y}{9 x^2+4 y^4}.
\end{align*}
Since $f$ is a regular frontal, the system of equations $da=b_1\,d\theta_1+b_2\,d\theta_2$ has a unique solution,
\begin{align*}
	b_1(x,y)=\frac{x y^2 \left(9 x^2 y^2+4 y^6+10\right)}{\sqrt{9 x^2+4 y^4} \sqrt{9 x^2 y^2+4 y^6+4}}; && b_2(x,y)=\frac{3 x^2 y-2 y^5}{\sqrt{9 x^2 y^2+4 y^6+4}}
\end{align*}
The vector fields $\mu_1, \mu_2\colon \mb{R}^2 \to S^2$ completing the orthonormal basis $\{\nu(x,y),\mu_1(x,y),\mu_2(x,y)\}$ for $T_{f_n(x,y)}\mb{R}^3$ are given by
\begin{align*}
	\mu_1(x,y)=-\frac{(4y^3,6 x y, 9 x^2y^2+4y^6)}{\sqrt{9 x^2 y^2+4 y^6} \sqrt{9 x^2 y^2+4 y^6+4}};
		&&
	\mu_2(x,y)=\left(\frac{3 x y}{\sqrt{9 x^2 y^2+4 y^6}},-\frac{2 y^3}{\sqrt{9 x^2 y^2+4 y^6}},0\right)
\end{align*} 
Then we have that
	\[a(x,y)\nu(x,y)+b_1(x,y)\mu_1(x,y)+b_2(x,y)\mu_2(x,y)=(x,y^2,xy^3)=f(x,y)\]
and $f$ can be recovered from the Legendre data $(\nu,a)$.
\end{example}

\begin{example}[Cross-cap]
Let $f\colon \mb{R}^2 \to \mb{R}^3$ be the Whitney umbrella,
	\[f(x,y)=(x,y^2,xy).\]
This is a pseudo regular frontal with $\Reg(f)=\Reg(\nu)=\mb{R}^2\backslash\{(0,0)\}$, but it is not a frontal.
The Legendre data associated to $f$ is
\begin{align*}
	\nu(x,y)=\frac{(-2 y^2,-x, 2 y)}{\sqrt{x^2+4 \left(y^4+y^2\right)}}; && a(x,y)=-\frac{x y^2}{\sqrt{x^2+4 \left(y^4+y^2\right)}}.
\end{align*}
Writing $\nu(x,y)=\nu^2(\theta_1(x,y),\theta_2(x,y))$ gives us functions
\begin{align*}
	\sin \theta_1(x,y)=&\frac{2 y}{\sqrt{x^2+4 \left(y^4+y^2\right)}}; 	&\cos \theta_1(x,y)=&\sqrt{\frac{x^2+4 y^4}{x^2+4 y^4+4 y^2}};\\
	\sin \theta_2(x,y)=&\frac{x}{\sqrt{x^2+4 y^4}}; 				&\cos \theta_2(x,y)=&-\frac{2 y^2}{\sqrt{x^2+4 y^4}}.
\end{align*}
Applying the chain rule, we see that
\begin{align*}
	\frac{\partial \theta_1}{\partial x}=\frac{1}{\cos \theta_1}\frac{\partial \sin \theta_1}{\partial x}&=-\frac{2 x y}{\sqrt{x^2+4 y^4} \left(x^2+4 \left(y^4+y^2\right)\right)};
	&
	\frac{\partial \theta_2}{\partial x}=\frac{1}{\cos \theta_2}\frac{\partial \sin \theta_2}{\partial x}&=\frac{2 y^2}{\sqrt{x^2+4 y^4} \sqrt{x^2+4 \left(y^4+y^2\right)}};
	\\
	\frac{\partial \theta_1}{\partial y}=\frac{1}{\cos \theta_1}\frac{\partial \sin \theta_1}{\partial y}&=\frac{2 \left(x^2-4 y^4\right)}{\sqrt{x^2+4 y^4} \left(x^2+4 \left(y^4+y^2\right)\right)};&
	\frac{\partial \theta_2}{\partial y}=\frac{1}{\cos \theta_2}\frac{\partial \sin \theta_2}{\partial y}&=-\frac{4 x y \left(x^2+2 y^2 \left(2 y^2+2\right)\right)}{\sqrt{x^2+4 y^4} \left(x^2+4 \left(y^4+y^2\right)\right)^{3/2}}.
\end{align*}
Since $f$ is a pseudo regular frontal, the system of equations $da=b_1\,d\theta_1+b_2\,d\theta_2$ has a unique solution,
\begin{align*}
	b_1(x,y)=\frac{x y \left(x^2+4 y^4+6 y^2\right)}{\sqrt{x^2+4 y^4} \sqrt{x^2+4 y^4+4 y^2}}; && b_2(x,y)=\frac{x^2-2 y^4}{\sqrt{x^2+4 y^4}}.
\end{align*}
The vector fields $\mu_1, \mu_2\colon \mb{R}^2 \to S^2$ completing the orthonormal basis $\{\nu(x,y),\mu_1(x,y),\mu_2(x,y)\}$ for $T_{f_n(x,y)}\mb{R}^3$ are given by
\begin{align*}
	\mu_1(x,y)=-\frac{(4y^3, 2 x y, x^2+4y^4)}{\sqrt{x^2+4 y^4} \sqrt{x^2+4 \left(y^4+y^2\right)}};
		&&
	\mu_2(x,y)=\left(\frac{x}{\sqrt{x^2+4 y^4}},-\frac{2 y^2}{\sqrt{x^2+4 y^4}},0\right)
\end{align*} 
Then we have that
	\[a(x,y)\nu(x,y)+b_1(x,y)\mu_1(x,y)+b_2(x,y)\mu_2(x,y)=(x,y^2,xy)=f(x,y)\]
and $f$ can be recovered from the Legendre data $(\nu,a)$.
\end{example}

\begin{example}[Cuspidal edge]
Let $f\colon \mb{R}^2 \to \mb{R}^3$ be the cuspidal edge, given by $f(x,y)=(x,y^2,y^3)$.
{The mapping $f$ is a frontal, but it is not a regular frontal; nonetheless, }$f$ is the pointwise limit of the sequence of mappings $f_n\colon \mb{R}^2 \to \mb{R}^3$ given by
	\[f_n(x,y)=\left(x,\frac{x^2}{n}+y^2,y^3\right),\]
which are regular frontals with $\Reg(f_n)=\Reg(\nu_n)=\{(x,y) \in \mb{R}^2: y\neq 0\}$.
Moreover, we can transform $f_n$ into $f$ for any $n \in \mb{N}$ by applying a suitable change of coordinates in the source and target.

Let $n \in \mb{N}$: the Legendre data associated to $f_n$ is
\begin{align*}
	\nu_n(x,y)=\frac{(6 x y,-3 n y, 2n)}{\sqrt{9 n^2 y^2+4 n^2+36 x^2 y^2}}; && a_n(x,y)=\frac{y \left(3 x^2-n y^2\right)}{\sqrt{9 n^2 y^2+4 n^2+36 x^2 y^2}}.
\end{align*}
Writing $\nu_n(x,y)=\nu^2(\theta_1(x,y),\theta_2(x,y))$ gives us functions
\begin{align*}
	\sin \theta_1(x,y)=&\frac{2 n}{\sqrt{9 n^2 y^2+4 n^2+36 x^2 y^2}}; 	&\cos \theta_1(x,y)=&3 y \sqrt{\frac{n^2+4 x^2}{9 n^2 y^2+4 n^2+36 x^2 y^2}};\\
	\sin \theta_2(x,y)=&-\frac{n}{\sqrt{n^2+4 x^2}}; 					&\cos \theta_2(x,y)=&\frac{2 x}{\sqrt{n^2+4 x^2}}.
\end{align*}
{Note that when $n$ goes to infinity, $\sin \theta_2(x,y)$ goes to $1$ and $\cos\theta_2(x,y)$ goes to $0$ for any $(x,y) \in \Reg(f_n)$, meaning that $\theta_2$ must be a constant function and the differential equation
	\[da=b_1\,d\theta_1+b_2\,d\theta_2\]
admits infinite solutions $(b_1,b_2)$.}

Applying the chain rule, we see that
\begin{align*}
	\frac{\partial \theta_1}{\partial x}=\frac{1}{\cos \theta_1}\frac{\partial \sin \theta_1}{\partial x}&=-\frac{24 n x y}{\sqrt{n^2+4 x^2} \left(9 n^2 y^2+4 n^2+36 x^2 y^2\right)};
	&
	\frac{\partial \theta_2}{\partial x}=\frac{1}{\cos \theta_2}\frac{\partial \sin \theta_2}{\partial x}&=\frac{2 n}{n^2 + 4 x^2};
	\\
	\frac{\partial \theta_1}{\partial y}=\frac{1}{\cos \theta_1}\frac{\partial \sin \theta_1}{\partial y}&=-\frac{6 n \sqrt{n^2+4 x^2}}{n^2 \left(9 y^2+4\right)+36 x^2 y^2};&
	\frac{\partial \theta_2}{\partial y}=\frac{1}{\cos \theta_2}\frac{\partial \sin \theta_2}{\partial y}&=0.
\end{align*}
Since $f_n$ is a regular frontal, the system of equations $da=b_1\,d\theta_1+b_2\,d\theta_2$ has a unique solution,
\begin{align*}
	b_1(x,y)=\frac{3 n^2 y^4+2 n^2 y^2-2 n x^2+12 x^2 y^4}{\sqrt{n^2+4 x^2} \sqrt{9 n^2 y^2+4 n^2+36 x^2 y^2}}; && b_2(x,y)=\frac{3 x y \left(n^2+2 n y^2+2 x^2\right)}{n \sqrt{9 n^2 y^2+4 n^2+36 x^2 y^2}}
\end{align*}
The vector fields $\mu_1, \mu_2\colon \mb{R}^2 \to S^2$ completing the orthonormal basis $\{\nu_n(x,y),\mu_1(x,y),\mu_2(x,y)\}$ for $T_{f_n(x,y)}\mb{R}^3$ are given by
\begin{align*}
	\mu_1(x,y)=-\frac{(-4 n x, 2n^2, 3y(n^2+4x^2)}{\sqrt{n^2+4 x^2} \sqrt{9 n^2 y^2+4 n^2+36 x^2 y^2}};
		&&
	\mu_2(x,y)=\left(\frac{n}{\sqrt{n^2+4 x^2}},\frac{2 x}{\sqrt{n^2+4 x^2}},0\right)
\end{align*} 
Then we have that
	\[a_n(x,y)\nu_n(x,y)+b_1(x,y)\mu_1(x,y)+b_2(x,y)\mu_2(x,y)=\left(x,\frac{x^2}{n}+y^2,y^3\right)=f_n(x,y)\]
and $f$ can be recovered from the sequence of Legendre data $\{(\nu_n,a_n)\}_{n \in \mb{N}}$.
\end{example}

\begin{example}[Swallowtail]
Let $f\colon \mb{R}^2 \to \mb{R}^3$ be the swallowtail singularity, given by $f(x,y)=(x,2y^3+xy,3y^4+xy^2)$.
{The mapping $f$ is a frontal, but it is not a regular frontal; nonetheless, }$f$ is the pointwise limit of the sequence of mappings $f_n\colon \mb{R}^2 \to \mb{R}^3$ given by
	\[f_n(x,y)=\left(x,x y+2 y^3,\frac{x^2}{n}+x y^2+3 y^4\right),\]
which are regular frontals with $\Reg(\nu_n)=\Reg(f_n)=\{(x,y) \in \mb{R}^2: x\neq -6y^2\}$.
Moreover, we can transform $f_n$ into $f$ for any $n \in \mb{N}$ by applying a suitable change of coordinates in the source and target.

Let $n \in \mb{N}$: the Legendre data associated to $f_n$ is
\begin{align*}
	\nu_n(x,y)=\frac{(ny^2-2x,2ny, n)}{\sqrt{4 n^2 y^2+n^2+\left(2 x-n y^2\right)^2}}; && a_n(x,y)=-\frac{n y^4+x^2}{\sqrt{n^2 y^4+4 n^2 y^2+n^2-4 n x y^2+4 x^2}}.
\end{align*}
Writing $\nu_n(x,y)=\nu^2(\theta_1(x,y),\theta_2(x,y))$ gives us functions
\begin{align*}
	\sin \theta_1(x,y)=&\frac{n}{\sqrt{4 n^2 y^2+n^2+\left(2 x-n y^2\right)^2}}; 		&\cos \theta_1(x,y)=&\sqrt{\frac{n^2 y^4+4 n^2 y^2-4 n x y^2+4 x^2}{n^2 y^4+4 n^2 y^2+n^2-4 n x y^2+4 x^2}};\\
	\sin \theta_2(x,y)=&-\frac{2 n y}{\sqrt{n^2 y^4+4 n^2 y^2-4 n x y^2+4 x^2}}; 	&\cos \theta_2(x,y)=&\frac{n y^2-2 x}{\sqrt{n^2 y^4+4 n^2 y^2-4 n x y^2+4 x^2}}.
\end{align*}
Applying the chain rule, we see that
\begin{align*}
	\frac{\partial \theta_1}{\partial x}&=-\frac{2 n \left(2 x-n y^2\right)}{\sqrt{n^2 y^4+4 n^2 y^2-4 n x y^2+4 x^2} \left(n^2 y^4+4 n^2 y^2+n^2-4 n x y^2+4 x^2\right)};
	\\
	\frac{\partial \theta_2}{\partial x}&=-\frac{4 n y}{n^2 y^4+4 n^2 y^2-4 n x y^2+4 x^2};
	\\
	\frac{\partial \theta_1}{\partial y}&=\frac{2 n^2 y \left(-n y^2-2 n+2 x\right)}{\sqrt{n^2 y^4+4 n^2 y^2-4 n x y^2+4 x^2} \left(n^2 y^4+4 n^2 y^2+n^2-4 n x y^2+4 x^2\right)};\\
	\frac{\partial \theta_2}{\partial y}&=\frac{2 n \left(n y^2+2 x\right)}{n^2 y^4+4 n^2 y^2-4 n x y^2+4 x^2}.
\end{align*}
Since $f_n$ is a regular frontal, the system of equations $da=b_1\,d\theta_1+b_2\,d\theta_2$ has a unique solution,
\begin{align*}
	b_1(x,y)&=\frac{n^3 y^2 \left(x y^4+4 x y^2+x+3 y^6+12 y^4+4 y^2\right)+n^2 x \left(x \left(-3 y^4+4 y^2+2\right)-12 y^6\right)+12 n x^2 y^4+4 x^4}{n \sqrt{n^2 y^2 \left(y^2+4\right)-4 n x y^2+4 x^2} \sqrt{n^2 \left(y^4+4 y^2+1\right)-4 n x y^2+4 x^2}};\\
	b_2(x,y)&=\frac{y \left(n x y^2+2 n x+2 n y^4-2 x^2-4 x y^2\right)}{\sqrt{n^2 y^4+4 n^2 y^2+n^2-4 n x y^2+4 x^2}}
\end{align*}
The vector fields $\mu_1, \mu_2\colon \mb{R}^2 \to S^2$ completing the orthonormal basis $\{\nu_n(x,y),\mu_1(x,y),\mu_2(x,y)\}$ for $T_{f_n(x,y)}\mb{R}^3$ are given by
\begin{align*}
	\mu_1(x,y)&=\frac{(n(2x-ny^2),2n^2y,n^2 y^4+4 n^2 y^2-4 n x y^2+4 x^2)}{\sqrt{n^2 y^4+4 n^2 y^2-4 n x y^2+4 x^2} \sqrt{4 n^2 y^2+n^2+\left(2 x-n y^2\right)^2}};
		\\
	\mu_2(x,y)&=\frac{(2 n y,n y^2-2 x)}{\sqrt{n^2 y^4+4 n^2 y^2-4 n x y^2+4 x^2}}.
\end{align*} 
Then we have that
	\[a_n(x,y)\nu_n(x,y)+b_1(x,y)\mu_1(x,y)+b_2(x,y)\mu_2(x,y)=\left(x,x y+2 y^3,\frac{x^2}{n}+x y^2+3 y^4\right)=f_n(x,y)\]
and $f$ can be recovered from the sequence of Legendre data $\{(\nu_n,a_n)\}_{n \in \mb{N}}$.
\end{example}

\begin{figure}[ht]
	\includegraphics[width=.8\textwidth]{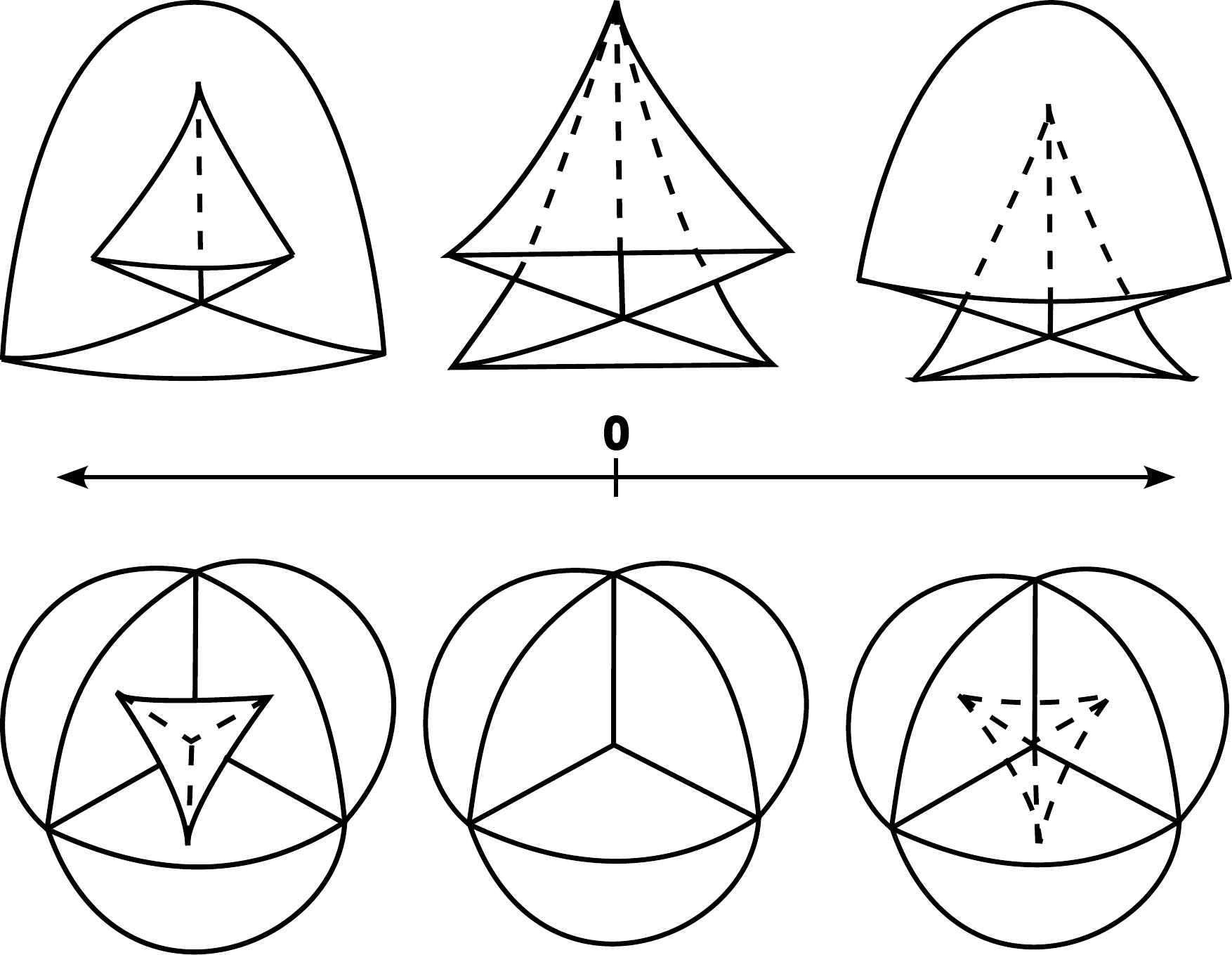}
	\caption{Image of the $D_4^+$ (top) and $D_4^-$ (bottom) singularities, projected into $\R^3$ via the map $(u,x,y,z) \mapsto (x,y,z)$ for different values of $u$. \label{D4 fig}}
\end{figure}

\begin{example}[$D_4$ singularity]
Let $f\colon \mb{R}^3 \to \mb{R}^4$ be the $D_4^+$ singularity, given by $f(u,v,w)=(u,v w,2 u v+3 v^2+w^2,u v^2+2 v^3+2 v w^2)$ (see Figure \ref{D4 fig}).
{The mapping $f$ is a frontal, but it is not a regular frontal; nonetheless, }$f$ is the pointwise limit of the sequence of mappings $f_n\colon \mb{R}^3 \to \mb{R}^4$ given by
	\[f_n(x,y)=\left(u,v w,2 u v+3 v^2+w^2,\frac{u^2}{n}+u v^2+2 v^3+2 v w^2\right),\]
which are regular frontals with $\Reg(\nu_n)=\Reg(f_n)=\{(u,v,w) \in \mb{R}^3: w^2\neq uv+3v^2\}$.
Moreover, we can transform $f_n$ into $f$ for any $n \in \mb{N}$ by applying a suitable change of coordinates in the source and target.

Let $n \in \mb{N}$: the Legendre data associated to $f_n$ is
\begin{align*}
	\nu_n(u,v,w)=\frac{(nv^2-2u,-2nw, -nv,n)}{\sqrt{n^2 v^2+4 n^2 w^2+n^2+\left(n v^2-2 u\right)^2}}.
\end{align*}
Writing $\nu_n(u,v,w)=\nu^3(\theta_1(u,v,w),\theta_2(u,v,w),\theta_3(u,v,w))$ gives us functions
\begin{align*}
	\sin \theta_1=	&\frac{n}{\sqrt{n^2 v^2+4 n^2 w^2+n^2+\left(n v^2-2 u\right)^2}}; &
	\cos \theta_1=	&\sqrt{\frac{n^2 v^4+n^2 v^2+4 n^2 w^2-4 n u v^2+4 u^2}{n^2 v^4+n^2 v^2+4 n^2 w^2+n^2-4 n u v^2+4 u^2}};\\
	\sin \theta_2=	&-\frac{n v}{\sqrt{n^2 v^4+n^2 v^2+4 n^2 w^2-4 n u v^2+4 u^2}};&
	\cos \theta_2=	&\frac{\sqrt{n^2 v^4+4 n^2 w^2-4 n u v^2+4 u^2}}{\sqrt{n^2 v^4+n^2 v^2+4 n^2 w^2-4 n u v^2+4 u^2}};\\
	\sin \theta_3=	&-\frac{2 n w}{\sqrt{n^2 v^4+4 n^2 w^2-4 n u v^2+4 u^2}};&
	\cos \theta_3=	&\frac{n v^2-2 u}{\sqrt{n^2 v^4+4 n^2 w^2-4 n u v^2+4 u^2}}.
\end{align*}
Applying the chain rule, we see that
\begin{align*}
	\frac{\partial \theta_1}{\partial u}&=-\frac{2 n \left(2 u-n v^2\right)}{\sqrt{n^2 v^4+n^2 v^2+4 n^2 w^2-4 n u v^2+4 u^2} \left(n^2 v^4+n^2 v^2+4 n^2 w^2+n^2-4 n u v^2+4 u^2\right)};
	\\
	\frac{\partial \theta_2}{\partial u}&=-\frac{2 n v \left(n v^2-2 u\right)}{\sqrt{n^2 v^4+4 n^2 w^2-4 n u v^2+4 u^2} \left(n^2 v^4+n^2 v^2+4 n^2 w^2-4 n u v^2+4 u^2\right)};
	\\
	\frac{\partial \theta_3}{\partial u}&=-\frac{4 n w}{n^2 \left(v^4+4 w^2\right)-4 n u v^2+4 u^2};
	\\
	\frac{\partial \theta_1}{\partial v}&=-\frac{n^2 v \left(2 n v^2+n-4 u\right)}{\sqrt{n^2 v^4+n^2 v^2+4 n^2 w^2-4 n u v^2+4 u^2} \left(n^2 v^4+n^2 v^2+4 n^2 w^2+n^2-4 n u v^2+4 u^2\right)};
	\\
	\frac{\partial \theta_2}{\partial v}&=-\frac{n \left(-n^2 v^4+4 n^2 w^2+4 u^2\right)}{\sqrt{n^2 v^4+4 n^2 w^2-4 n u v^2+4 u^2} \left(n^2 v^4+n^2 v^2+4 n^2 w^2-4 n u v^2+4 u^2\right)};
	\\
	\frac{\partial \theta_3}{\partial v}&=\frac{4 n^2 v w}{n^2 \left(v^4+4 w^2\right)-4 n u v^2+4 u^2};
	\\
	\frac{\partial \theta_1}{\partial w}&=-\frac{4 n^3 w}{\sqrt{n^2 v^4+n^2 v^2+4 n^2 w^2-4 n u v^2+4 u^2} \left(n^2 v^4+n^2 v^2+4 n^2 w^2+n^2-4 n u v^2+4 u^2\right)};
	\\
	\frac{\partial \theta_2}{\partial w}&=\frac{4 n^3 v w}{\sqrt{n^2 v^4+4 n^2 w^2-4 n u v^2+4 u^2} \left(n^2 v^4+n^2 v^2+4 n^2 w^2-4 n u v^2+4 u^2\right)};
	\\
	\frac{\partial \theta_3}{\partial w}&=-\frac{2 n \left(n v^2-2 u\right)}{n^2 \left(v^4+4 w^2\right)-4 n u v^2+4 u^2}.
\end{align*}
Since $f_n$ is a regular frontal, the system of equations $da=b_1\,d\theta_1+b_2\,d\theta_2+b_3\,d\theta_3$ has a unique solution given by
\begin{align*}
	b_1(u,v,w)=&\frac{n^3 v \left(u \left(v^5+v^3+4 v w^2+v\right)+2 v^6+2 v^4 \left(w^2+1\right)+v^2 \left(10 w^2+3\right)+w^2 \left(8 w^2+3\right)\right)}{n \sqrt{n^2 \left(v^4+v^2+4 w^2\right)-4 n u v^2+4 u^2} \sqrt{n^2 \left(v^4+v^2+4 w^2+1\right)-4 n u v^2+4 u^2}}+\\
		&+\frac{n^2 u \left(u \left(-3 v^4+v^2+4 w^2+2\right)-8 v^3 \left(v^2+w^2\right)\right)+8 n u^2 v \left(v^2+w^2\right)+4 u^4}{n \sqrt{n^2 \left(v^4+v^2+4 w^2\right)-4 n u v^2+4 u^2} \sqrt{n^2 \left(v^4+v^2+4 w^2+1\right)-4 n u v^2+4 u^2}};\\
	b_2(u,v,w)=&\frac{n^2 \left(u \left(2 v^5+v^3+8 v w^2\right)+3 v^6+v^4 w^2+10 v^2 w^2+4 w^4\right)}{\sqrt{n^2 \left(v^4+4 w^2\right)-4 n u v^2+4 u^2} \sqrt{n^2 \left(v^4+v^2+4 w^2+1\right)-4 n u v^2+4 u^2}}+\\
			&+\frac{-2 n u v \left(4 u v^2+u+6 v^3+2 v w^2\right)+4 u^2 \left(2 u v+3 v^2+w^2\right)}{\sqrt{n^2 \left(v^4+4 w^2\right)-4 n u v^2+4 u^2} \sqrt{n^2 \left(v^4+v^2+4 w^2+1\right)-4 n u v^2+4 u^2}};\\
	b_3(u,v,w)=&\frac{w \left(n \left(2 u+v^3\right)-2 u v\right)}{\sqrt{n^2 \left(v^4+v^2+4 w^2+1\right)-4 n u v^2+4 u^2}}.
\end{align*}
The vector fields $\mu_1, \mu_2,\mu_3\colon \mb{R}^3 \to S^3$ completing the orthonormal basis $\{\nu_n,\mu_1,\mu_2,\mu_3\}$ for $T_{f_n(u,v,w)}\mb{R}^4$ are given by
\begin{align*}
	\mu_1(u,v,w)&=\frac{(n(nv^2-2u),2n^2w,n^2v,n^2 v^4+n^2 v^2+4 n^2 w^2-4 n u v^2+4 u^2)}{\sqrt{n^2 v^4+n^2 v^2+4 n^2 w^2-4 n u v^2+4 u^2} \sqrt{n^2 v^2+4 n^2 w^2+n^2+\left(n v^2-2 u\right)^2}};
		\\
	\mu_2(u,v,w)&=\frac{(n v (n v^2-2 u),2n^2 vw,n^2 v^4+4 n^2 w^2-4 n u v^2+4 u^2,0)}{\sqrt{n^2 v^4+4 n^2 w^2-4 n u v^2+4 u^2} \sqrt{n^2 v^4+n^2 v^2+4 n^2 w^2-4 n u v^2+4 u^2}};
\\
	\mu_3(u,v,w)&=\left(\frac{2 n w}{\sqrt{n^2 v^4+4 n^2 w^2-4 n u v^2+4 u^2}},\frac{n v^2-2 u}{\sqrt{n^2 v^4+4 n^2 w^2-4 n u v^2+4 u^2}},0,0\right).
\end{align*}
Then we have that
\begin{align*}
	&a_n(x,y)\nu_n(x,y)+b_1(x,y)\mu_1(x,y)+b_2(x,y)\mu_2(x,y)+b_3(x,y)\mu_3(x,y)
	\\
	&=\left(u,v w,2 u v+3 v^2+w^2,\frac{u^2}{n}+u v^2+2 v^3+2 v w^2\right)=f_n(x,y)
\end{align*}
and $f$ can be recovered from the sequence of Legendre data $\{(\nu_n,a_n)\}_{n \in \mb{N}}$.
\end{example}

\vspace{5cm}
\section*{Acknowledgements}
This work was completed during the visit 
of the second author to Universitat de Val\`{e}ncia.
He would like to thank Universitat de Val\`{e}ncia for their kind hospitality.    
{\color{black}The authors would like to thank Armando Pérez 
at Universitat de Val\`{e}ncia for his kind explanation on  
quantum entanglements. }     
\par 
This work was supported 
by the Research Institute for Mathematical Sciences, 
a Joint Usage/Research Center located in Kyoto University. 
\par 
The first and third authors were partially supported by Grant
PID2021-124577NB-I00 funded by MCIN/AEI/ 10.13039/501100011033 
and by “ERDF A way of making 
Europe”.
The second author was partially 
supported by JSPS KAKENHI (Grant No. 23K03109).

%
%
%
%

\end{document}